\documentclass[a4paper]{scrarticle}
\usepackage{amscd,amssymb,amsmath,amsthm}
\usepackage[utf8]{inputenc}
\usepackage[greek,english]{babel}
\DeclareUnicodeCharacter{2207}{$\nabla$}
\DeclareUnicodeCharacter{03C6}{$\varphi$}
\DeclareUnicodeCharacter{03D5}{$\phi$}
\usepackage{graphicx}
\usepackage{booktabs}
\usepackage{color}
\usepackage[lofdepth,lotdepth]{subfig}
\usepackage[backend=bibtex,style=ieee,giveninits=true,sorting=anyt]{biblatex}
\usepackage{hyperref}
\usepackage{todonotes}
\newcommand{\Z}{\mathbb{Z}}
\newtheorem{thm}{Theorem}
\newtheorem{defn}{Definition}
\newtheorem{lemma}{Lemma}
\newtheorem{pro}{Proposition}
\newtheorem{rk}{Remark}

\DeclareMathOperator\arcosh{arcosh}
\newcommand{\footnoteremember}[2]{
\footnote{#2}
\newcounter{#1}
\setcounter{#1}{\value{footnote}}
}
\newcommand{\footnoterecall}[1]{
\footnotemark[\value{#1}]
}

\begin{document}
\title{Existence of gradient Gibbs measures on regular trees which are not translation invariant} 
\author{Florian Henning\footnoteremember{RUB}{Ruhr-Universit\"at Bochum, Fakult\"at f\"ur Mathematik, D44801 Bochum, Germany} \footnote{Florian.Henning@ruhr-uni-bochum.de  }
	\and Christof K\"ulske \footnoterecall{RUB} \footnote{Christof.Kuelske@ruhr-uni-bochum.de, 
		\newline
		\url{https://www.ruhr-uni-bochum.de/ffm/Lehrstuehle/Kuelske/kuelske.html}}\, 
	\,  
}
		
		\maketitle
	
\begin{abstract} We provide an existence theory for gradient Gibbs measures for 
$\Z$-valued spin models on regular trees which are not invariant under translations of the tree, assuming only summability 
of the transfer operator. 
The gradient states we obtain are delocalized. 
The construction we provide for them starts from a two-layer hidden Markov model representation in a setup which is not invariant under tree-automorphisms,   
involving internal $q$-spin models.  
The proofs of existence and lack of translation invariance of infinite-volume gradient states are based on properties of the local pseudo-unstable manifold of the corresponding discrete dynamical systems of these internal models, around the free state, at large $q$. 
\end{abstract}

boundary law, heavy tails, stable manifold theorem.

\section{Introduction}	
The question whether statistical mechanics models with translation-invariant 
interactions allow infinite-volume states which are not translation invariant has a long history. 
A famous example which shows that this is possible are the Dobrushin-states 
for the Ising model in zero external field on the integer lattice $\Z^d$, in dimensions
$d\geq 3$. They can be 
obtained with plus/minus boundary conditions on the upper/lower half of a sequence of cubes. 
At sufficiently low temperatures they break translation invariance, see \cite{Do73}
and also \cite{BrLePf79}. 
By contrast, in low lattice dimensions $d\leq 2$ such states do not exist in the Ising model,  for, all Gibbs states of the Ising model are necessarily translation-invariant, 
see \cite{Ai80},\cite{Hi81},\cite{CoVe12} and also \cite{CovELNRu18}.

More generally speaking, this is just an example for the following type of question: Consider an equilibrium statistical mechanics model 
on a regular graph, with a graph-automorphism invariant Hamiltonian (specification).  
When are there Gibbs measures for this Hamiltonian which break some 
symmetries of the underlying graph?   

We will be interested in the following in the case where the graph is a regular tree. 
Let us therefore start by recalling what is known for the ferromagnetic Ising model in zero external field: 
In that case there are even uncountably many automorphism non 
invariant Gibbs measures in the full low-temperature region 
(Theorem 12.31 in \cite{Ge11}). 
This region is equivalently described as the region for which 
$\mu^+\neq \mu^-$ where the latter states are the graph-automorphism invariant measures
obtained as finite-volume limits with homogeneous 
plus respectively minus boundary conditions.   
Nonhomogeneous states of different type exist in regions of even lower temperatures, see \cite{GaMaRuSh20} and also \cite{Ro13}.

In this paper we plan to investigate the analogous question 
for integer-valued \textit{gradient models} on regular trees, which are described 
in terms of a spatially homogeneous nearest neighbor interaction potential 
\begin{equation*}U:\Z \rightarrow \mathbb{R} \end{equation*}
which assigns the energetic contribution $U(\omega_x-\omega_y)$ associated to an 
edge with endpoints $x,y$ which is felt by an integer-valued spin configuration. 
For such gradient models we want to ask specifically 
for the possibility or impossibility of existence of
automorphism non invariant \textit{gradient states}.
Similarly to the situation on the lattice (\cite{FS97},\cite{Sh05}) the Gibbs property of these gradient states means compatibility with the kernels of the \textit{Gibbsian gradient specification} associated with the height-shift invariant interaction potential.

Note that gradient states, as distributions on increments of the variables,  
may exist in regimes where Gibbs states do not exist. This is the case in the class 
of translation-invariant (t.i.) states for convex nearest neighbor potentials for real variables 
on the two-dimensional integer lattice \cite{FS97}. 
Regarding the $d$-regular tree (\textit{Cayley tree}), by the term \textit{translation} we mean a tree-automorphism which is obtained by considering the tree as a group acting on itself (see Section \ref{subsec: BasicDef} for the precise definition).

So, we stress that the question about automorphism non invariant (or non-t.i.)
gradient states  has to be distinguished from the analogous question about
automorphism non invariant (or non-t.i.) Gibbs states.  

For more on gradient models on the lattice with homogeneous interactions we refer to \cite{CoDeuMu09}, \cite{DeuGiIo00}, \cite{Sh05}, \cite{BiSp11} and \cite{BoCiKu17}. 
For more on gradient models on the lattice in random environments, see \cite{BK94}, \cite{EnKu08}, \cite{CKu12},\cite{DaHaPe21}.

What is known on trees for non-trivial  
automorphism-invariant gradient states that is those with dependent increments, which do not arise as projections to the increments of the variables in any Gibbs measure? In recent work  \cite{HKLR19}, \cite{HK21} 
the existence of such states which are different from the free state, was investigated on the $d$-regular tree, $d \geq 2$. 
The authors found conditions on the transfer operator \begin{equation*}Q=e^{-\beta U},\end{equation*} formulated in terms 
of smallness of the pair of the $\frac{d+1}{2}$- and the $d+1$-norm of $Q$ on $\Z$, which ensured existence of gradient Gibbs measures (GGM) different from the free state with i.i.d.-increments. Here, the choice of norms appeared from a contraction argument based on Young's inequality for convolutions. The established condition on the smallness of $p$-norms was fullfilled when the inverse temperature $\beta$ was chosen sufficiently large. In particular, a $\frac{\log d}{d}$-asymptotics for $\beta$ as $d$ tends to infinity was proved.
\subsection{Main statement}
In the present paper we prove in full generality the following main result which informally reads: 

\textit{Gradient models for any summable strictly positive transfer operator $Q$ always possess 
non-t.i. gradient states, see Theorem \ref{thm: main}}. 

We stress that our theorem makes no assumption on 
convexity or monotonicity of the interaction, only summability and $Q(i)>0$ for all $i \in \mathbb{Z}$. 
As there is always the free state of i.i.d.- increments, which is automorphism invariant, this implies that there is 
never uniqueness of GGMs for any fixed summable $Q$, 
which seems surprising indeed at first look.  

How can one expect such a  very general existence result to hold, 
	and how to prove it, 
	without any further properties of $Q$? 
\subsection{Ideas of the proof}
The framework of the proof is the construction of integer-valued gradient Gibbs measures from the Gibbs measures of underlying $q$-spin \textit{clock models} which have a discrete rotational symmetry. Here, $q=2,3,\ldots$ takes the meaning of a period. This construction, outlined in Section \ref{subsubsec: fuzzyTogradient}, is done in terms of a two-step procedure, where for any realization of the underlying $Z_q$-valued Gibbs measure integer-valued increments are edge-wise independently sampled using an appropriate stochastic kernel built from the interaction potential (see Theorem \ref{thm: DLR-eq}). The relation between $q$-spin Gibbs measures and integer-valued gradient Gibbs measures generalizes an earlier result of \cite{HK21}, Section 3.1, which was restricted to the case of tree-automorphism invariant gradient Gibbs measures. More precisely, in \cite{HK21}, for the proof of the Gibbs property of the constructed gradient measures, we referred to \cite{KS17} where the same class of tree-automorphism invariant gradient Gibbs measures was obtained in terms of a different construction based on mixing over pinned gradient measures. The proof of the Gibbs property we give in Theorem \ref{thm: DLR-eq} of this paper is shorter, more straightforward, and also covers the case of gradient Gibbs measures which are not invariant under the automorphisms of the tree.
%In the first step one understands in Theorem \ref{thm: DLR-eq} that 
%{\color{orange}gradient Gibbs measures (\textit{GGM})} can be obtained as stochastic images of {\color{orange} the Gibbs measures of} internal 
%$q$-spin models which have a discrete rotation invariance, also called clock-models, 
%where $q$ takes the meaning of a period which may run through the integers $2,3,\dots$. 
%
%In this relation (possibly non-automorphism invariant) 
%Gibbs measures of the $q$-spin models get mapped to GGMs for the gradient Hamiltonian. 
%{\color{orange}
%This type of relation was already exploited in \cite{HK21} in the automorphism-invariant setup, 
%but will be shown to hold also in the non-automorphism invariant setup, by a new method of proof,}
%see Section \ref{subsubsec: fuzzyTogradient}. 
We describe in the following Subsection \ref{subsec: delocalization} 
why GGMs obtained as images in this way are necessarily 
delocalized (i.e. they do not stem from a Gibbs measure of the initial model.

In the second main part we investigate these internal $q$-spin models at fixed $q$. This is done via the description of Gibbs measures in terms of 
\textit{boundary law solutions}, based on Zachary's theory \cite{Z83}.
We look specifically for solutions which are not translation invariant but radially symmetric
(around an arbitrary root). 
The corresponding boundary law formalism can be naturally interpreted 
in terms of a  $q$-indexed family of discrete-time dynamical system 
\begin{equation*}
S_q:\Delta^q\rightarrow \Delta^q
\end{equation*}
on the simplex $\Delta^q$, see the equations \eqref{eq: RecIn} and \eqref{eq: RecAway}, where 
the non-linear map $S_q$ 
to be analyzed  depends on the interaction $Q$ and the period $q$.

\subsubsection*{Backwards trajectories, stability analysis, and GGMs.}
Automor\-phism non invariant states of the $q$-clock models are then obtained via their correspondence to backwards trajectories of the map $S_q$.
The difficulty with this argument is that such infinite backwards iterates may not always exist, 
for arbitrary model parameters, periods, and initial values. This turns the existence problem 
for such GGMs in general highly nontrivial. Indeed, a general understanding of these dynamical systems, structures of fixed points,  
and their bifurcations in their dependence on $Q$, 
other than for very small values of $q$, poses challenging model-dependent tasks. For specific work on aspects of inhomogeneity
of solutions for Ising and Potts models, see 
%For work on inhomogeneous solutions in Ising and Potts models from a a different angle, see
 \cite{GaMaRuSh20}, \cite{PeSt99}, \cite{BiEnEn17}, \cite{PePe10} and \cite{Sl11}.		
It turns out {to }be very fruitful for our model-independent approach to focus on an important common property 
of the maps $S_q$, shared by all gradient models for summable $Q$, and periods $q$.   
Indeed, for any period $q$, 
in any parameter regime there is always the free state, which for any $q$ corresponds to the trivial 
fixed point of $S_q$ provided by the equidistribution. The crucial point is, 
that depending on $q$, at fixed $Q$, the stability  
properties of this fixed point change, 
and this has useful consequences for the existence of automorphism non invariant states. Namely, we show that for any summable 
$Q$, when the  period $q\geq q_0(Q)$ is large enough, then  
there is an unstable manifold of 
positive dimension of the map $S_q$ around the equidistribution. In this case 
non-trivial backwards trajectories are obtained from starting points on this 
unstable manifold away from the fixed point, 
and they yield the desired states. 

There is a part of the argument where we need care to be able to deal also 
with the exceptional cases of non-hyperbolicity (which may occur 
at exceptional parameter values), and this is where we invoke the pseudo-unstable 
manifold theorem of \cite{Ch02}. This yields a countable family of distinct non-trivial measures indexed by $q$, see Proposition \ref{pro: Id}.
		
By contrast to this general existence theorem, obtained for large enough periods $q$, 
the other interesting related question after existence of non-trivial solutions 
at fixed periods $q$ requires specific properties of the spectrum of the transfer 
operator $Q$, see Theorem \ref{thm: existence}. Moreover, starting from Dobrushin's uniqueness Theorem we show that at fixed $q$ these non-trivial solutions indeed fail to exist at sufficiently high temperatures, see Proposition \ref{pro: smallq}.
For a quantitative discussion of this in the context of the SOS-model as well as for a particular polynomially decaying transfer operator, see Theorem \ref{thm: EB} in Section \ref{sec: Applications}. 
		
\subsubsection*{Proving lack of translation invariance of the associated GGMs.}
To complete the proof of existence of GGMs which are not invariant under translations we need in a final step 
to understand how to get from a backwards trajectory of $S_q$ to the associated GGM we are 
finally interested in.
We show via a local argument beyond linearization 
on the local pseudo-unstable manifold  (see Section \ref{Sec: non-translation})
that the automorphism non invariance 
of the radially symmetric $q$-periodic boundary law solutions we construct 
really survives the map to the corresponding gradient state.  
In fact, we obtain even non-\textit{translation-invariance} when the Cayley tree is viewed as group acting 
on itself,  which is stronger. 

The paper is organized as follows. 
Section \ref{sec: Def} contains the definitions of the model, and of the basic notions 
of Gibbs measures, gradient Gibbs measures, and tree-indexed Markov chains. 
Section \ref{sec: constr_GGM} describes the map from period-$q$ boundary
laws to GGMs, allowing for inhomogeneity.  
Section \ref{Sec: Existence theory for non-invariant GGMs} contains our existence results for non-invariant GGMs, 
which are explained in terms of backwards trajectories of the maps $S_q$.
Section \ref{sec: Applications} illustrates the theory for two prototypical models.

Finally, the proofs are given in Section \ref{sec: Proofs}.
\subsection*{Acknowledgements} 
Florian Henning is partially supported by the Research Training Group 2131 \textit{High-dimensional phenomena in probability-Fluctuations and discontinuity} of German Research Council (DFG).
%\newpage 
\section{Definitions} \label{sec: Def}
\subsection{Height configurations, Markov chains and translations on the Cayley tree}\label{subsec: BasicDef}
We consider models on the Cayley-tree $\Gamma^d=(V,L)$ of order $d\geq2$ with the integers $\Z$ as local state space and denote the set of \textit{height-configurations} $\Z^V$ by $\Omega$. Let $\Z$ be equipped with the $\sigma$-algebra given by its power set and for any $\Lambda \subset V$ let $\sigma_\Lambda:\Z^V \rightarrow \Z^\Lambda, (\omega_x)_{x \in V} \mapsto (\omega_x)_{x \in \Lambda} $ denote the coordinate spin projection to the spins inside $\Lambda$. Then we consider the measurable space $(\Omega,\mathcal{F})$ where $\mathcal{F}:=\sigma(\sigma_{\{x\}} \mid x \in V)$ is the product-$\sigma$-algebra. For any subvolume $\Lambda \subset V$ we denote by $\mathcal{F}_\Lambda:=\sigma(\sigma_{\{x\}} \mid x \in \Lambda )$ the $\sigma$-algebra on $\Omega$ generated by the spins inside the volume $\Lambda$.

The term \textit{Cayley tree} (or $d$-regular tree) means a connected graph without cycles where each vertex has exactly $d+1$ nearest neighbors.
We call two vertices $x,y \in V$ \textit{nearest neighbors} if they are connected by an edge $b=\{x,y\} \in L$.

A collection of $n$ edges $\{x,x_1\},\{x_1,x_2\}, \ldots, \{x_{n-1},y\}$ is called a \textit{path} (of length $n$) from $x$ to $y$, whereas an infinite collection of nearest neighbor pairs will be called an \textit{infinite path}.
For two vertices $x,y$ the distance $d(x,y)$ is defined as the length of the shortest path from $x$ to $y$.
Besides the set of unoriented edges $L$ which contains two-element subsets of $V$ we also consider the set $\vec{L}$ of oriented edges which contains ordered pairs of vertices. Hence, $(x,y) \in \vec{L}$ denotes an oriented edge, while $\{x,y\} \in L$ denotes the respective unoriented edge, which we notationally distinguish to emphasize at which steps orientation of an edge is relevant.
If we restrict to a connected subset $\Lambda \subset V$ and set $L_\Lambda:=\{ \{x,y\} \in L \mid x,y \in \Lambda \}$ then $(\Lambda, L_\Lambda)$ is the subtree of $\Gamma$ with vertices inside $\Lambda$. Similarly, $\vec{L}_\Lambda:=\{ (x,y) \in \vec{L} \mid x,y \in \Lambda \}$ for any subset $\Lambda \subset V$.

Furthermore, for any $\Lambda \subset V$ we define its outer boundary by
\begin{equation*}
\partial \Lambda := \{ x \notin \Lambda : d(x,y) = 1 \mbox{ for some } y \in \Lambda\}.
\end{equation*}
If $\Lambda \subset V$ is finite then we write $\Lambda \Subset V$.
As outlined in Chapter 1.2 of \cite{Ro13}, the Cayley tree $\Gamma^d=(V,L)$ of order $d$ as a planar graph can be represented by the free product $G_d$ of $d+1$ cyclic groups of second order (i.e. groups which contain exactly two elements). Every element in $G_d$ is a finite word of symbols where each two adjacent symbols are from different groups. 

The group representation of $V$ is then obtained as follows: Fix any root $\rho \in V$. Then $\rho$ is represented by the unit $e \in G_d$. The $d+1$ nearest neighbors are enumerated counter-clockwise by the symbols $a_1,\ldots a_{d+1}$. Now for any $n \geq 1$ let $v_n$ be any vertex at distance $n$ to the root. Then $v_n$ has a unique nearest neighbor $v_{n-1}$ lying on the shortest path from $\rho$ to $v_n$. $v_{n-1}$ is represented by a word of length $n-1$. Let its rightmost symbol be $a_j$. Then the group representation of $v_n$ is obtained by adding a symbol different from $a_j$ on the right. Adding the symbol $a_j$ again one gets back to $v_{n-2}$. This enumeration of the $d+1$ nearest neighbors to $v_{n-1}$ in terms of the symbols $a_1,\ldots,a_{d+1}$ conveniently corresponds to the counter-clockwise ordering in a planar embedding of the graph.
Two words are multiplied by concatenation and reduction.

A useful concept in the context of \textit{tree-indexed Markov chains} is the notion of \textit{future} and \textit{past} of a vertex which can be also found in chapter 12 of \cite{Ge11}.
Given any vertex $v \in V$, we write
\begin{equation}
\vec{L}^v:=\{(x,y) \in \vec{L} \mid d(v,x)=d(v,y)+1\}
\end{equation}
for the set of edges pointing towards $v$ and
\begin{equation}
{}^v\vec{L}:=\{(x,y) \in \vec{L} \mid d(v,y)=d(x,v)+1\}
\end{equation}
for the set of edges pointing away from $v$.
Then 
\begin{equation}
\begin{split}
%]xy,\infty[ \ &:= \ \{v \in V \mid (x,y) \in \vec{L}^v\} \quad \text{and} \cr 
]-\infty,xy[ \ &:= \ \{v \in V \mid (x,y) \in {}^v\vec{L}\}
\end{split}
\end{equation}
denotes the past of the oriented edge $(x,y) \in \vec{L}.$
Using this notation, a \textit{tree-indexed Markov chain} on $(\Omega, \mathcal{F})$ is a probability measure $\mu$ such that for any $(x,y) \in \vec{L}$
\begin{equation*}
\mu(\sigma_{y}=i \mid  \mathcal{F}_{]-\infty,xy[})=\mu(\sigma_{y}=i \mid \mathcal{F}_x) \quad \mu-\text{a.s.}
\end{equation*}
\subsection{Gibbs measures and gradient Gibbs measures}
On the space of height-configurations we consider a symmetric nearest-neighbor interaction potential $\Phi$ with corresponding \textit{transfer operator} $Q$ defined by 

\begin{equation*}
Q_b(\zeta):=\exp\left(-\Phi_b(\zeta) \right)
\end{equation*}
for any edge $b=\{x,y\} \in L$ and $\zeta \in \mathbb{Z}^b$.

The kernels of the Gibbsian specification $(\gamma_\Lambda)_{\Lambda \Subset V}$ then read 
\begin{equation} \label{Def: Gibbs specification}
\gamma_\Lambda(\sigma_\Lambda=\omega_\Lambda \mid \omega)=Z_\Lambda(\omega_{\partial \Lambda})^{-1} \prod_{b \cap \Lambda \neq \emptyset}Q_b(\omega_b).
\end{equation}
A \textit{Gibbs measure} for the specification $\gamma$ is a probability measure $\mu$ on $(\Omega,\mathcal{F})$ such that for all finite $\Lambda \subset V$ and any $A \in \mathcal{F}$
\begin{equation}
\mu(A \mid \mathcal{F}_{\Lambda^c})=\gamma_\Lambda(A \mid \cdot ) \ \mu\text{-a.s.}	
\end{equation}
This is equivalent to $\mu \gamma_\Lambda=\mu$ for any finite $\Lambda \subset V$.

We denote the set of Gibbs measures by $\mathcal{G}(\gamma)$.

In this paper we focus on the special case of symmetric gradient interactions, i.e. for any edge $b=\{x,y \} \in L$
\begin{equation} \label{Def: GradPot}
Q_b(\omega_x,\omega_y)=Q_b(\omega_x-\omega_y)=\exp(-\beta U_b(\omega_x-\omega_y)),
\end{equation}
where the parameter $\beta>0$ will be regarded as inverse temperature and each $U_b: \Z \rightarrow [0,\infty)$ is a symmetric function.

We are insterested in the particular case of measures which are invariant under joint translations of the local state space $\Z$.
Given any height-configuration $\omega=(\omega_x)_{x \in V}$ we define the respective \textit{gradient configuration} $\nabla \omega \in \mathbb{Z}^{\vec{L}}$ by setting $(\nabla \omega)_{(x,y)}:=\omega_y-\omega_x$ for any edge $(x,y) \in \vec{L}$. Clearly, \begin{equation}\label{eq: GradientConstraint}(\nabla \omega)_{(x,y)}=-(\nabla \omega)_{(y,x)} \text{ for all } (x,y) \in \vec{L}. \end{equation} In the other direction, from connectedness of the tree and absence of loops it follows that a given gradient configuration satisfying the symmetry constraint \eqref{eq: GradientConstraint} and prescription of the height at a fixed vertex defines a unique height-configuration.
Hence the set  \[\Omega^{\nabla}=\{(\zeta_{(x,y)})_{(x,y) \in \vec{L}} \mid \zeta_{(x,y)}=-\zeta_{(y,x)} \text{ for all } (x,y) \in \vec{L} \}\] of gradient configurations bijectively corresponds to the quotient $\Z^V / \Z$, the set of relative heights.
For any subset $\Lambda \subset V$ we denote by $\eta_\Lambda: \Omega^\nabla \rightarrow \Z^{\vec{L}_\Lambda}$ the gradient spin projection to edges with both vertices inside the volume $\Lambda$. Equip $\Omega^\nabla$ with the product-$\sigma$-algebra $\mathcal{F}^\nabla=\sigma(\eta_{(x,y)} \mid (x,y) \in \vec{L})$. For any $\Lambda$, we set $\mathcal{F}_\Lambda^\nabla=\sigma(\eta_{(x,y)} \mid (x,y) \in \vec{L}_\Lambda)$.
% denote the $\sigma$-algebra on $\Omega^\nabla$ generated by all gradient {\color{blue} spin projections to edges} with both vertices \textbf{inside} the volume $\Lambda$. 
 By construction, for any finite connected $\Lambda \subset V$ the $\sigma$-algebra $\mathcal{F}_\Lambda^\nabla$ can be identified with the set of all events in $\mathcal{F}_\Lambda$ which are invariant under joint height-shift of all spins.

In this paper we are interested in probability measures on the space of gradient configurations which are Gibbs in the sense that they are invariant under the respective gradient configuration for the transfer operator $Q$.
We note that due to the absence of cycles the complement of any (finite) subtree $(\Lambda,L_\Lambda)$ of $\Gamma$ decomposes into distinct connected components. This means that information on the gradients outside of $\Lambda$ does not determine a relative height-configuration on $\Lambda^c$, by which we understand an element of $\Z^{\Lambda^c} / \Z$. Hence an event $A \in \mathcal{F}_{\Lambda^c}$ which is invariant under joint height-shift at all sites is in general not measurable with respect to $\mathcal{F}^\nabla_{\Lambda^c}$.
Therefore we will introduce a further \textit{outer} $\sigma$-algebra $\mathcal{T}^\nabla_\Lambda$ which incorporates both the gradients outside the subtree $(\Lambda, L_\Lambda)$ and the relative heights at the boundary, by which we understand an element of $\mathbb{Z}^{\partial \Lambda} / \mathbb{Z}$.
More precisely, if we fix any $(\Lambda,L_\Lambda)$, any vertex $x \in \partial \Lambda$ and any absolute height $\omega_x \in \Z$ then any gradient configuration $\zeta \in \Omega^{\nabla}$ gives rise to a unique height configuration on $\Z^{\partial \Lambda}$ which depends only on the values of the gradient spin variables inside $\Lambda \cup \partial \Lambda$. This follows from connectedness of the subtree  $(\Lambda, L_\Lambda)$. Hence we obtain an $\mathcal{F}^\nabla_{\Lambda \cup \partial \Lambda}$-measurable function $[\eta]_\Lambda:\Omega^\nabla \rightarrow \Z^{\partial \Lambda} /\Z$. Here, the set $\Z^{\partial \Lambda}$ is endowed with the product-$\sigma$-algebra and $\Z^{\partial \Lambda} /\Z$ is endowed with the $\sigma$-algebra generated by the projection.
Then $\mathcal{T}^\nabla_\Lambda$ is given by
\begin{equation}\label{eq: OuterAlg}
\mathcal{T}^\nabla_\Lambda=\sigma((\eta_{(x,y)})_{(x,y) \in {\vec{L}_{\Lambda^c}}}, \ [\eta]_\Lambda).
\end{equation}
Now the gradient Gibbs specification $(\gamma'_\Lambda)_{\Lambda \Subset V}$ associated to a Gibbsian specification $(\gamma_\Lambda)_{\Lambda \Subset V}$ is defined as follows (see \cite{KS17}):
\begin{defn}\label{def: GradientGibbs}
	Consider the outer $\sigma$-algebra $\mathcal{T}^\nabla_\Lambda$ (see \eqref{eq: OuterAlg}) 
	Then the \textit{gradient Gibbs specification}
	is defined as the family of probability kernels $(\gamma'_\Lambda)_{\Lambda \Subset V}$ from $(\Omega^\nabla, \mathcal{T}_\Lambda^\nabla)$ to $(\Omega^\nabla, \mathcal{F}^\nabla)$ given by
	\begin{equation}\label{grad}
	\int F(\rho) \gamma'_\Lambda(\text{d}\rho \mid \zeta) = \int F(\nabla \varphi) \gamma_\Lambda(\text{d}\varphi\mid\omega)
	%\gamma'_{\L}(F \mid \zeta)=Q (F(\eta) \mid {\eta \sim_{\partial \L} \zeta}).
	\end{equation}
	for all bounded $\mathcal{F}^{\nabla}$-measurable functions $F$, where $\omega \in \Omega$ is any height-configuration with $\nabla \omega = \zeta$.
\end{defn}
Having defined the gradient specification we conclude with giving the definition of a \textit{gradient Gibbs measure} (see \cite{KS17}).
\begin{defn} A measure $\nu \in \mathcal{M}_1(\Omega^\nabla)$ is called a \textit{gradient Gibbs measure (GGM)} if it satisfies the DLR equation 
\begin{equation}\label{eq: DLR}
\int \nu (d\zeta)F(\zeta)=\int \nu (\text{d}\zeta) \int \gamma'_{\Lambda}(\text{d}\tilde\zeta \mid \zeta) F(\tilde\zeta)
\end{equation}
for every finite $\Lambda \subset V$ and for all bounded continuous functions $F$ on $\Omega^\nabla$.  
\end{defn}
This is equivalent to
\begin{equation}
\nu(A \mid \mathcal{T}^\nabla_\Lambda)=\gamma^\prime_\Lambda(A \mid \cdot) \quad \nu \text{-a.s.}
\end{equation}
for all  $A \in \mathcal{F}^\nabla$ and all finite $\Lambda \subset V$.
\section{Two-layer hidden Markov model construction in automorphism non invariant setup} \label{sec: constr_GGM}
In this section we give a general construction for gradient Gibbs measures.
These gradient Gibbs measures will be constructed from (possibly spatially inhomogeneous) height-periodic functions satisfying an appropriate version of Zachary's \cite{Z83} boundary law equation. The results of this section are a generalization of the homogeneous case considered in \cite{KS17} and \cite{HK21}, where the result of \cite{KS17} is restricted to the construction of automorphism-invariant GGMs.  
Note that, while the rest of this paper is focused on spatially homogeneous interaction potentials, in this section we regard the interaction potential (the transfer operator, respectively) as a possibly spatially dependent object, as it allows to track back terms in the proofs. Furthermore this also gives - without much extra effort- the opportunity to refer to these results in future research not necessarily restricted to spatially homogeneous interactions.
\subsection{Background on relation between boundary laws and Gibbs measures}
The marginals of a Gibbs measure for a nearest-neighbor potential in a finite subtree $(\Lambda, L_\Lambda)$ can be written as the product of the associated transfer operator evaluated at the spins at the edges with at least one vertex in the subtree and some $\mathcal{F}_{\partial \Lambda}$-measurable function.
By Zachary \cite{Z83} this function can be expressed by so-called boundary laws and one obtains a one-to-one relation between boundary laws and those Gibbs measures which are also tree-indexed Markov chains, covering the class of extremal Gibbs measures. We cite the theorem in its original form allowing also for nonhomogeneous transfer operators. We will use it later only for homogeneous interactions, but nonhomogeneous boundary law solutions. 
\begin{defn}\label{def:bl}
A family of functions $\{ \lambda_{xy} \}_{( x,y ) \in \vec L}$ with $\lambda_{xy}:=\lambda_{(x,y)} \in [0, \infty)^{\Z}$ and $\lambda_{xy}  \not \equiv 0$ is called a \textbf{boundary law} for the family of transfer operators $\{ Q_b\}_{b \in L}$ if 
\begin{itemize} 
\item[i)  ]for each $( x,y ) \in \vec L$ there exists a constant  $c_{xy}>0$ such that the \textbf{boundary law equation}
\begin{equation}\label{eq: bl}
\lambda_{xy}(\omega_x) = c_{xy} \prod_{z \in \partial x \setminus \{y \}} \sum_{\omega_z \in \mathbb{Z}} Q_{zx}(\omega_x,\omega_z) \lambda_{zx}(\omega_z)
\end{equation}
holds for every $\omega_x \in \mathbb{Z}$ and 
\item[ii) ]for any $x \in V$ the \textbf{normalizability condition}
\begin{equation} \label{eq: normalizability}
\sum_{\omega_x \in \mathbb{Z}} \Big( \prod_{z \in \partial x} \sum_{\omega_z \in \mathbb{Z}} Q_{zx}(\omega_x,\omega_z) \lambda_{zx}(\omega_z) \Big) < \infty
\end{equation}
holds true.
\end{itemize}
\end{defn}
Then the associated theorem reads
\begin{thm}[Theorem 3.2 in \cite{Z83}] \label{thm: Zachary}
Let $(Q_b)_{b \in L}$ be any family of transfer operators such that there is some $\omega \in \Omega$ with \begin{equation} \label{cond: nonnull}
Q_{\{x,y \}}(i, \omega_y)>0 \text{ for all } \{x,y\} \in L \text{ and any }i \in \Z.
\end{equation}
	
Then for the Markov specification $\gamma$ associated to $(Q_b)_{b \in L}$  we have:
\begin{itemize}
\item[i)  ] Each boundary law $(\lambda_{xy})_{( x,y ) \in \vec{L}}$ for $(Q_b)_{b \in L}$ defines a unique tree-indexed Markov chain Gibbs measure $\mu \in \mathcal{G}(\gamma)$ with marginals
\begin{equation}\label{BoundMC}
\mu(\sigma_{\Lambda \cup \partial \Lambda}=\omega_{\Lambda \cup \partial \Lambda}) = (Z_\Lambda)^{-1} \prod_{y \in \partial \Lambda} \lambda_{y y_\Lambda}(\omega_y) \prod_{b \cap \Lambda \neq \emptyset} Q_b(\omega_b),
\end{equation}
for any connected set $\Lambda \Subset V$ where $y \in \partial \Lambda$, $y_\Lambda$ denotes the unique $n.n.$ of $y$ in $\Lambda$ and $Z_{\Lambda}$ is the normalization constant which turns the r.h.s. into a probability measure.
\item[ii) ]Conversely, every tree-indexed Markov chain Gibbs measure $\mu \in \mathcal{G}(\gamma)$ admits a representation of the form (\ref{BoundMC}) in terms of a boundary law (unique up to a constant positive factor).
\end{itemize}
\end{thm}
In the context of gradient potentials \eqref{Def: GradPot}, the requirement \eqref{cond: nonnull} means strict positivity of the family $(Q_b)_{b \in L}$. 
One approach in reducing complexity of the system of equations \eqref{eq: bl} for gradient potentials \eqref{Def: GradPot} is assuming that all occurring functions $\lambda_{xy} \in [0,\infty)^\Z$ are periodic functions on $\Z$ for some common period $q \in \{2,3,\ldots\}$.
\begin{rk}
In \cite{Z83}, boundary laws are formally defined as equivalence classes of families of functions, two functions being equivalent if and only if one is obtained by multiplying the other one by  a suitable edge-dependent positive constant. Equivalent representatives of a boundary law are associated to the same tree-indexed Markov chain Gibbs measure. Hence we can impose a further constraint to select a representative, e.g. by fixing the value of an arbitrary norm to be one, as it will be done in the following.   
\end{rk} 
\subsection{Height-periodic boundary laws}
In what follows we assume that the transfer operator $\{Q_b\}_{b \in L}$ is summable, i.e. $\sum_{i \in \Z} \vert Q_b(i) \vert<\infty$ for all $b \in L$.
 
We call a family  $(\lambda^q_{xy})_{(x,y) \in \vec{L}}$ of functions $\lambda^q_{xy} \in [0, \infty)^{\Z}$ a $q$-\textit{height-periodic boundary law} for the transfer operator $\{ Q_b\}_{b \in L}$ if it solves the boundary law equation \eqref{eq: bl} and for any $(x,y) \in \vec{L}$ the function $\lambda^q_{xy}: \mathbb{Z} \rightarrow (0, \infty)$ is $q$-periodic. 
Clearly, there is a one-to-one correspondence between the set of $q$-height-periodic boundary laws for a transfer operator $\{ Q_b\}_{b \in L}$ and the set of boundary laws on the finite local state space $\mathbb{Z}_q=\mathbb{Z} / q\mathbb{Z}$ for the associated \textit{fuzzy transfer operator} $(Q_b^q)_{b \in L}$ where $Q_b^q(\bar{i}):= \sum_{j \in \bar{i}=i+q\Z}Q_b(j), \bar{i} \in \mathbb{Z}_q$. 
One direction is simply given by setting $\lambda^q_{xy}(\bar{i})=\lambda^q_{xy}(j)$ for any $j \in i+q\mathbb{Z}$.
Hence, a $q$-height periodic boundary law (normalized to be an element of the unit simplex $\Delta^q$) can be computed by solving the finite-dimensional system of equations 
\begin{equation} \label{eq: perGenBL}
\lambda^q_{xy}(\bar{i})=\frac{\prod_{z \in \partial \{x\} \setminus y} \sum_{\bar{j} \in \mathbb{Z}_q}Q_{\{x,z\}}^q(\bar{i} - \bar{j})\lambda^q_{zx}(\bar{j})}{\Vert \prod_{z \in \partial \{x\} \setminus y} \sum_{\bar{j} \in \mathbb{Z}_q}Q_{\{x,z\}}^q(\bar{\cdot} - \bar{j})\lambda^q_{zx}(\bar{j}) \Vert_1}, \quad \bar{i} \in \mathbb{Z}_q.
\end{equation}
at any edge $(xy) \in \vec{L}$. 

Note that finiteness of all $Q^q_b$ is equivalent to summability of all $Q_b$ which explains the necessity of this assumption.
\subsection{A two-step construction of gradient Gibbs measures}
In what follows, we assume that $q \in \{2,3, \ldots \}$ and  $(Q_b)_{b \in L}$ is any symmetric transfer operator of the form \eqref{Def: Gibbs specification} such that for any $b \in L$ the fuzzy transfer operator $Q_b^q$ is strictly positive. 

Height-periodic boundary laws do not fulfill the normalizability condition \eqref{eq: normalizability}. This means that the measure formally given by \eqref{BoundMC} is not defined.
\subsubsection*{The $q$-spin fuzzy chain}
However, employing a two-step procedure by which we first sample a tree-indexed Markov chain, the so-called \textit{fuzzy chain} on $(\Z_q^V,\mathcal{F}^q)$ according to \eqref{BoundMC}, where $\mathcal{F}^q=\otimes_{x \in V} 2^{\Z_q}$ is the product-$\sigma$-algebra generated by the fuzzy-spin projections $\bar{\sigma}_{\{x\}}: \Z_q^V \rightarrow \Z_q$, we are able to construct gradient measures on $(\Omega^\nabla,\mathcal{F}^\nabla)$ which satisfy the DLR-equation \eqref{eq: DLR}. Here, $2^{\Z_q}$ denotes the power set on $\Z_q$. The fuzzy chain itself is an element of the set of Gibbs measures with respect to the \textit{fuzzy specification} $\gamma^q$
whose kernels are given by 
\begin{equation} 
\gamma^q_\Lambda(\bar{\sigma}_\Lambda=\bar{\omega}_\Lambda \mid \bar{\omega})=Z^q_\Lambda(\bar{\omega}_{\partial \Lambda})^{-1} \prod_{b \cap \Lambda \neq \emptyset}Q^q_b(\bar{\omega}_b).
\end{equation} for any finite $\Lambda \subset V$. After sampling a fuzzy chain, we independently apply kernels to the edges each describing a distribution of total increments (as elements of $\Z$) along an edge given the increment of the fuzzy chain (which is an element of $\Z_q$) along the respective edge. The so obtained probability  measure on the space of gradients $(\Omega^\nabla, \mathcal{F}^\nabla)$ is thus a hidden Markov model where the transition mechanism is not acting on the site-variables as it is more common, but acting on edge variables.

We will first present the two steps of construction and then give the precise definition of the respective DLR-equation and prove that the measure constructed solves them.  
\subsubsection*{From $q$-spin fuzzy increments to $\mathbb{Z}$-valued increments} \label{subsubsec: fuzzyTogradient}
For any $\{x,y \} \in L$ define a kernel $\rho_{Q_{\{x,y\}}}^{q}$ from $\Z_q$ to $\Z$ equipped with the power set $2^\mathbb{Z}$ by setting:
\begin{equation} \label{eq: RWPEtr}
\rho_{Q_{\{x,y\}}}^{q}(j \mid \bar{s})= \chi(j \in \bar{s})\frac{Q_{\{x,y\}}(j)}{Q_{\{x,y\}}^q(\bar{s})},
\end{equation} 
where $\chi$ denotes the indicator function.

Then define a map 
$T^q_Q: \mathcal{M}_1(\Z_q^V) \rightarrow \mathcal{M}_1(\Omega^\nabla, \mathcal{F}^\nabla)$ from spin-$q$ measures on vertices to gradient measures by setting
\begin{equation} \label{eq: ConstructionOfGGM}
T^q_Q(\mu^q)(\eta_\Lambda=\zeta_\Lambda)=\sum_{\bar{\omega}_\Lambda \in \Z_q^\Lambda}\mu^q(\bar{\sigma}_\Lambda=\bar{\omega}_\Lambda)\prod_{(x,y) \in {}^w\vec{L}, \, x,y \in \Lambda} \rho^q_{Q_{\{x,y\}}}(\zeta_{(x,y)} \mid \bar{\omega}_y-\bar{\omega}_x) 
\end{equation}
where $\Lambda \subset V$ is any finite connected set and $w \in V$ is an arbitrary site. This describes independent sampling over the edges with weight given by $Q$ conditional on the increment class.

We note that due to the fact that $\rho_{Q_{\{x,y\}}}^{q}(-j \mid -\bar{s})=\rho_{Q_{\{x,y\}}}^{q}(j \mid \bar{s})$ for all $\{x,y\} \in L$ and any $\bar{s} \in \Z_q$, $j \in \Z$, the definition of the map $T^q_Q$ does not depend on the concrete choice of the vertex $w$.

Given any $q$-periodic boundary law $(\lambda^q_{xy})_{(x,y) \in \vec{L}}$ we will write \begin{equation} \label{Not: BL_GGM}
\nu^{\lambda^q}:=T^q_Q(\mu^{\lambda^q})
\end{equation} where $\mu^{\lambda^q} \in \mathcal{G}(\gamma^q)$ is the fuzzy chain on $\mathbb{Z}_q^V$ associated to $(\lambda^q_{xy})_{(x,y) \in \vec{L}}$ by Theorem \ref{thm: Zachary}.

As the measures $\nu^{\lambda^q}$ will be of particular interest, we will write down a further marginals representation. 
\begin{lemma} \label{lem: oldRepOfGGM}
Let $\Gamma(x,y) \subset \vec{L}$ denote the set of edges in the shortest path between vertices $x$ and $y$ and $\bar{i} \in \Z_q$ the mod-$q$ projection of an integer $i$. Then the gradient measure $\nu^{\lambda^q}$ can be represented as
\begin{equation}
\nu^{\lambda^q}(\eta_{\Lambda \cup \partial \Lambda}=\zeta_{\Lambda \cup \partial \Lambda}) 
=Z_{\Lambda}^{-1} \left(\sum_{\bar{s} \in \Z_q} \prod_{y \in \partial \Lambda} \lambda^q_{yy_{\Lambda}}(\bar{s}+\sum_{b \in \Gamma(w,y)}\bar{\zeta_b}) \right) \prod_{b \cap \Lambda \neq \emptyset} Q_b(\zeta_b),
\end{equation}  
where $\Lambda \subset V$ is any finite connected volume, $w \in \Lambda$ is any fixed site and $Z_{\Lambda}$ is a normalization constant. Here, $y_\Lambda$ denotes the unique nearest neighbor of $y$ inside $\Lambda$.
\end{lemma}
%\hyperref[pr: lem1]{Go to proof}

The factor in parenthesis is the Radon-Nikodym derivative with respect to the free measure and indicates dependence.
Applying Lemma \ref{lem: oldRepOfGGM} to a singleton $\Lambda=\{x\}$ for any $x \in V$ and taking the boundary law equation \eqref{eq: bl} into account then gives in particular:
\begin{lemma} \label{lem: BondMarginal}
Let $(x,y) \in \vec{L}$ be an oriented edge. Then we have for the single-edge marginal of the gradient measure
\begin{equation*}
\nu^{\lambda^q}(\eta_{(x,y)}=\zeta_{(x,y)}) =Z^{-1}_{(x,y)}Q_{\{x,y\}}(\zeta_{(x,y)})\sum_{\bar{s} \in \mathbb{Z}_q}\lambda^q_{xy}(\bar{s})\lambda^q_{yx}(\bar{s}+\zeta_{(x,y)})
\end{equation*}
where $\zeta_{(x,y)} \in \mathbb{Z}$.
\end{lemma} 
%\hyperref[pr: lem2]{Go to proof}

Lemma \ref{lem: BondMarginal} will be used later to prove lack of translation invariance.
\subsection{The gradient Gibbs property}
The following result on the gradient Gibbs property of the image of the map $T_Q^q$ applied to Gibbs measures on $\Z_q$ is an extension of earlier results of \cite{KS17} and \cite{HK21} to non-homogeneous gradient Gibbs measures.
%First, we formally define a gradient specification $\gamma'$.
%Afterwards, we will state Theorem \ref{thm: DLR-eq} on the gradient Gibbs property of the image of the map $T_Q^q$ applied to Gibbs measures on $\Z_q$.
%\begin{defn}\label{def: GradientGibbs}
%Consider the outer $\sigma$-algebra $\mathcal{T}^\nabla_\Lambda$ (see \eqref{eq: OuterAlg}) 
%Then the \textit{gradient Gibbs specification}
%is defined as the family of probability kernels $(\gamma'_\Lambda)_{\Lambda \Subset V}$ from $(\Omega^\nabla, \mathcal{T}_\Lambda^\nabla)$ to $(\Omega^\nabla, \mathcal{F}^\nabla)$ given by
%\begin{equation}\label{grad}
%\int F(\rho) \gamma'_\Lambda(\text{d}\rho \mid \zeta) = \int F(\nabla \varphi) \gamma_\Lambda(\text{d}\varphi\mid\omega)
%%\gamma'_{\L}(F \mid \zeta)=Q (F(\eta) \mid {\eta \sim_{\partial \L} \zeta}).
%\end{equation}
%for all bounded $\mathcal{F}^{\nabla}$-measurable functions $F$, where $\omega \in \Omega$ is any height-configuration with $\nabla \omega = \zeta$.
%\end{defn}
%Our first structural result is the following Theorem \ref{thm: DLR-eq}, which also covers the non-homogeneous case.
\begin{thm} \label{thm: DLR-eq}
The map $T^q_Q$ maps Gibbs measures on $\Z_q^V$ for the fuzzy specification $\gamma^q$ to gradient Gibbs measures for the gradient Gibbs specification $\gamma'$ \eqref{grad}.	
\end{thm}	
%\hyperref[pr: thm2]{\Subset to proof}	
\subsection{Delocalization of possibly nonhomogeneous gradient Gibbs measures for homogeneous specifications}\label{subsec: delocalization}
We show that the sum of increments of gradient Gibbs measures for spatially homogeneous transfer operators along infinite paths diverge, and thus the heights delocalize. 
The following result is an extension of our result (Thm.4 in \cite{HK21}) for the special case of tree-automorphism invariant gradient Gibbs measures to nonhomogeneous GGMs.
\begin{pro} \label{pro: deloc}
If $(\lambda_{xy}^q)_{(x,y) \in \vec{L}}$ is a $q$-height-periodic boundary law solution for a spatially homogeneous family of positive transfer operators $Q$ then the possibly nonhomogeneous gradient Gibbs measure $\nu^{\lambda^q}$ associated to it via \eqref{eq: ConstructionOfGGM} \textit{delocalizes} in the sense that $\nu^{\lambda^q}(W_n=k) \stackrel{n \rightarrow \infty}{\rightarrow} 0$ for any total increment $W_n$ along a path of length $n$ and any $k \in \mathbb{Z}$. 
\end{pro}
%\hyperref[pr: pro1]{Go to proof}
\section{Existence theory for non-invariant GGMs of arbitrarily large periods $q$ via pseudo-unstable manifold theorem} \label{Sec: Existence theory for non-invariant GGMs}
In this section we are investigating the possibility to apply backwards 
iteration of the boundary 
law equation \eqref{eq: perGenBL} for radially symmetric (with respect to an arbitrary fixed root) $q$-height-periodic boundary laws to a spatially homogeneous transfer operator $Q$.  
We will obtain the following surprising general result: 
For any summable strictly positive spatially homogeneous $Q$, for large enough period there are always non-t.i. GGMs. 
We stress that there is no assumption on the 
existence of different automorphism invariant GGMs.
\subsection{Main results}
Assuming radial symmetry of the boundary law $\lambda$ with respect to a vertex $\rho \in V$, the boundary law equation \eqref{eq: perGenBL} at any $x \in \partial \{\rho\}$ reads
\begin{equation} \label{eq: BLHadamart}
\lambda_{\rho x}=H_q(\lambda_{y\rho}):=\frac{ (Q^q\lambda_{y \rho})^{\odot d}}{\Vert (Q^q\lambda_{y \rho})^{\odot d} \Vert_1}
\end{equation}
where $y \in \partial\{\rho\} \setminus \{x\}$.
Here, we identified the fuzzy transfer operator $Q^q$ with the symmetric circulant matrix $\left(Q^q(\bar{i}-\bar{j})\right)_{\bar{i},\bar{j} \in \Z_q}$ and $Q^q\lambda_{y \rho}$ denotes the matrix product.
For any $v,w \in \Delta^q$ and any $s \in \mathbb{R}$, the vectors $v^{\odot s} \in [0,\infty)^q \setminus \{(0,\ldots, 0)\}$ and $v \odot w \in [0,\infty)^q \setminus \{(0,\ldots, 0)\}$ are defined as
\[(v^{\odot s})_i=(v_i)^s \text{ and } (v \odot w)_i=v_i \cdot w_i, \text{ where } i \in \{0,1, \ldots, q-1 \}.\]
We will refer to the conjunction $\odot$ as the \textit{Hadamard-product}. Note that by contrast to the notation e.g. used in \cite{PeSt99}, in this paper the $\odot$-conjunction does not necessarily give a normalized object.

To remove the outer exponent $d$ on the r.h.s. of \eqref{eq: BLHadamart}, we may consider the transformation $u=G_\frac{1}{d}(\lambda):=\frac{\lambda^{\odot \frac{1}{d}}}{\Vert \lambda^{\odot \frac{1}{d}} \Vert_1 }$ and define the composed map \begin{equation*} 
S_q:=G_\frac{1}{d} \circ H_q \circ G_d
\end{equation*}
from $\Delta^q$ to $\Delta^q$. 
This means that for any $ u \in \Delta^q$
\begin{equation}\label{eq: TheMapSq}
S_q(u)= \frac{Q^q u^{\odot d}}{\Vert Q^q \Vert_1 \Vert u^{\odot d} \Vert_1}.
\end{equation}
To be able to apply backwards iteration to the map $S_q$ we need to know its spectrum. The following proposition gives a full description of it in terms of the Fourier-transform \begin{equation}
\hat{Q}: [-\pi,\pi) \rightarrow \mathbb{R} \, ; \quad \hat Q(k)=\sum_{n\in \Z}Q(n)\cos(n k )
	\end{equation} of $Q$.
\begin{pro}\label{pro: Eigenvalues} 	
The differential $D S_q[eq]: T \Delta^q \rightarrow T \Delta^q$ of the map $S_q$ computed 
in the equidistribution $eq=(1/q,\ldots,1/q)$ has the $\lfloor \frac{q}{2} \rfloor$ eigenvalues 
\[\frac{d \hat Q(2\pi\frac{j}{q})}{ \hat Q (0)}\]
where $j\in \{1,\dots, \lfloor \frac{q}{2} \rfloor \}$.
\end{pro}
%\hyperref[pr: pro2]{Go to proof}

Taking into account continuity of the Fourier transform 
$k\mapsto\hat Q(k)$ in the real variable $k$ at $k=0$, Proposition \ref{pro: Eigenvalues} implies the following
\begin{lemma}\label{lem: minPer} Fix any summable $Q$ with strictly positive elements. 
For any degree $d$ there is a finite 
minimal period $q_0(d)$ such that for all $q\geq q_0(d)$ 
at least one eigenvalue of the linearization 
$D S_q[eq]$ on $T \Delta^q$ is larger than one in absolute value. 
	
More generally, for any degree $d$, and any finite dimension $u$ there is a minimal period $q_0(d,u)$ such that for all $q\geq q_0(d,u)$ 
at least $u$ eigenvalues of $D S_q[eq]$ on $T \Delta^q$ are larger than one in absolute value. 
\end{lemma}

%\begin{proof} This follows by the continuity of the Fourier transform 
%	$k\mapsto\hat Q(k)$ in the real variable $k$ at $k=0$ from the previous lemma. 
%\end{proof}
We would now like to apply the radially symmetric backwards iteration on the local 
unstable manifold at the equidistribution corresponding to the strictly expanding eigenvalues of $\text{D}S_q[eq]$. 
There is the problem that we may encounter also cases of neutral eigenvalues (see Figure \ref{Fig: Nonhyperbolic} below)
\[\Big\lvert \frac{d \hat Q(2\pi\frac{j}{q})}{ \hat Q (0)} \Big\rvert =1.\]
At such points the hyperbolicity of the fixed point fails, and 
the standard stable manifold theorem for discrete-time dynamical systems for hyperbolic 
fixed points is not available in general.  

We can however bypass this difficulty by employing the $\tau$-unstable manifold theorem (e.g. Thm. 1.2.2 in \cite{Ch02}) which covers the case where the spectrum of $\text{D}S_q[eq]$ is off a circle of radius $\tau>1$. This was pointed out to us by Alberto Abbondandolo.  

\begin{thm}[Theorem 1.2.2 in \cite{Ch02}\label{thm: Chaperon} applied to the $\mathcal{C}^{\infty}$-map $S_q$]
Assume there is some $\tau>1$ such that the spectrum of $\text{D}S_q[eq]$ is the disjoint union of two sets $\sigma_1 \subset \{\vert\lambda \vert<\tau\}$ and $\sigma_2 \subset \{\vert\lambda \vert>\tau\}$ and let $U \subset \text{T}_{eq}\Delta^q$ denote the $\text{D}S_q[eq]$-stable subspace such that the spectrum of the restriction $\text{D}S_q[eq]_{\vert U}$ is $\sigma_2$. 

Then there exists a unique germ of a $\mathcal{C}^\infty$-manifold $W_\tau=W_\tau(S_q)$ having the following properties:
\begin{itemize}
	\item $S_q(W_\tau) \subset W_\tau$,
	\item $\text{T}_{eq}W_\tau=U$,
	\item The restriction ${S_q}_{\vert W_\tau}$ is the germ of a $\mathcal{C}^\infty$-diffeomorphism and 
	\item for any $u \in W_\tau$, the backwards iteration $S_q^{-n}$ tends to the equidistribution not slower than $\tau^{-n}$, where distance is measured in any local chart.
\end{itemize}
A representative of $W_\tau$ will be called a local $\tau$-unstable (or pseudo-unstable) manifold for the map $\text{D}S_q$ near the equidistribution.
\end{thm}   
Employing Proposition \ref{pro: Eigenvalues} to decide when it is possible to apply Theorem \ref{thm: Chaperon} to construct a radially symmetric boundary law solution via backwards iteration of $S_q$ and then going over from boundary laws to gradient Gibbs measures via the Theorems \ref{thm: Zachary} and \ref{thm: DLR-eq} we arrive at the following Theorem \ref{thm: existence}. The details of the construction and the promised lack of translation invariance are given in the subsections below.
\begin{thm}\label{thm: existence} Fix any period $q$ and any degree $d \geq 2$. Suppose that there is a level $\tau>1$ for which the Fourier transform $\hat{Q}$ of the transfer operator $Q$ satisfies
\begin{itemize} \label{cond: existence}
\item[i)  ] $\lvert \hat Q(2\pi\frac{j}{q}) \rvert \neq \frac{\tau}{d}\hat Q (0)$ for all indices $j\in \{1,\dots, q-1 \}$ and
\item[ii) ] the strict inequality 
$\vert \hat Q(2\pi\frac{j}{q}) \vert >  \frac{\tau}{d}\hat Q (0) $
is satisfied for some index $j\in \{1,\dots, q-1 \}$.
\end{itemize}	
Then there are 
gradient Gibbs measures of period $q$ which are not translation invariant.   
They are constructed from the non-homogeneous 
radially symmetric boundary law solutions obtained 
from  backwards iteration on the local 
$\tau$-unstable manifold $W_\tau$ of the non-linear 
map $S_q$ around the equidistribution. 
\end{thm}
As a direct corollary of Theorem \ref{thm: existence} we obtain our main result. 
\begin{thm}\label{thm: main} For any summable $Q$ and any degree $d\geq 2$ there is a finite period $q_0(d)$ such that for all $q\geq q_0(d)$ there are non-t.i. gradient Gibbs measures of period $q$.  
\end{thm}
%\hyperref[pr: thm4]{Go to proof}

\begin{center} %$\beta \in [1.2,3]$ $\beta \in [1.9,2.6]$
\begin{figure}[h]
\centering
{\includegraphics[width=7.5cm]{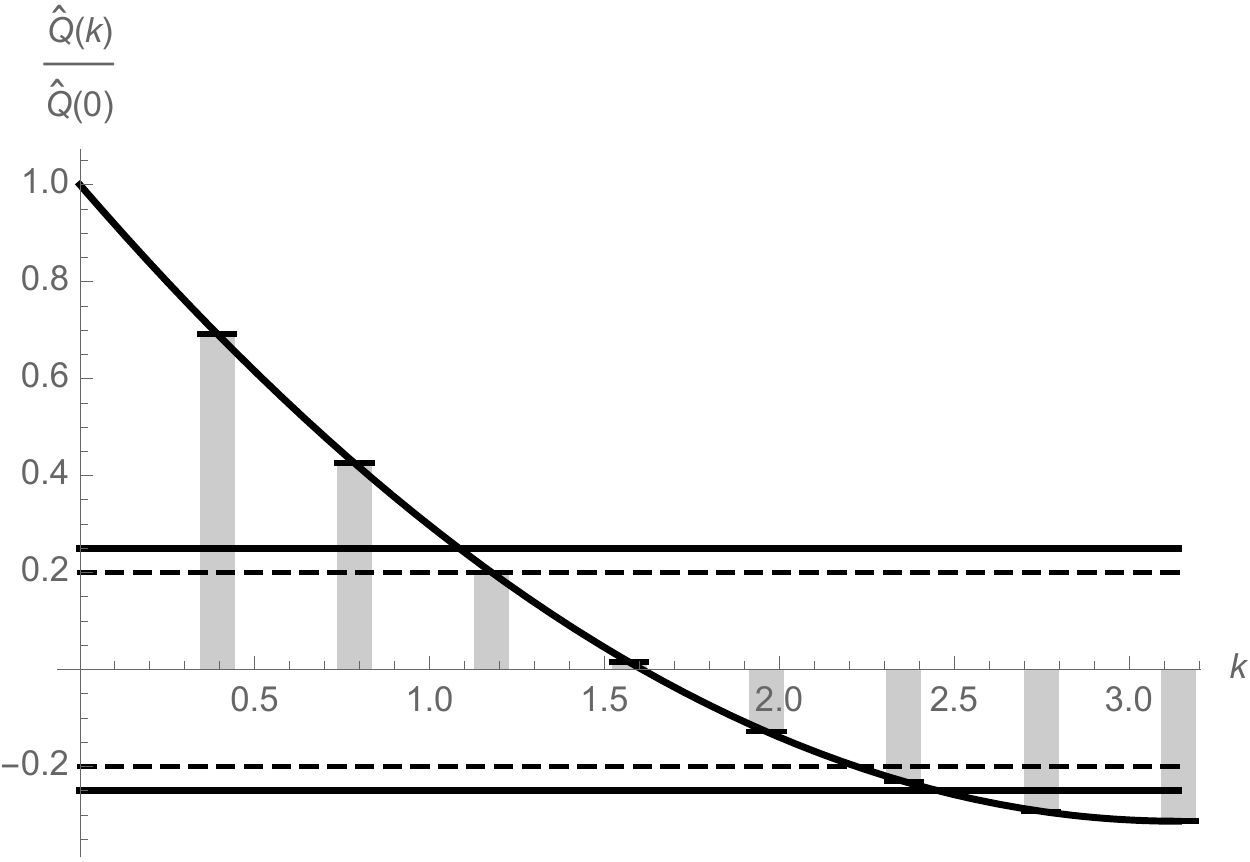}}
\caption{On the Cayley tree of order $d=5$, period $q=16$ and the inverse square model (see Section \ref{sec: Applications} below) the non-hyperbolic case occurs at an exceptional value of $a=\frac{1536}{73 \pi ^2} \approx 2.132$. The bars mark the values of $\hat{Q}(k)/\hat{Q}(0)$ at $k=j \frac{\pi}{8}$, where $j=1, \ldots,8$. The third bar from the left at $k=\frac{3\pi}{8}$ hits the upper dashed horizontal line marking the threshold $\frac{1}{d}=\frac{1}{5}$ and hence represents a neutral eigenvalue of $\text{D}S_q[eq]$. To satisfy the hypothesis of Theorem \ref{thm: existence}, we may put $\tau=\frac{5}{4}$ and slightly shift the dashed lines away from the horizontal axis to the solid horizontal lines at $\pm \frac{1}{4}$.} 
\label{Fig: Nonhyperbolic}
\end{figure}
\end{center}
\newpage
To visualize the map $S_q$ and the corresponding unstable manifold, we may consider the special case $q=2$ and represent the simplex $\Delta^2$ by the image of the map $[0,1] \ni s \mapsto (s,1-s)$. 
Then the first component of the map $S_2$ reads
\[S_2(s)=\frac{s^dQ^2(\bar{0})+(1-s)Q^2(\bar{1})}{\Vert Q^2 \Vert_1(s^d+(1-s)^d)}
%\begin{pmatrix}
%&s^dQ^2(\bar{0})+(1-s)Q^2(\bar{1}) \\
%&(1-s)^dQ^2(\bar{0})+s^dQ^2(\bar{1})
%\end{pmatrix}
\]
Dividing both the numerator and the denominator by $Q^2(\bar{1})$ we obtain the fixed point equation
\[s=\frac{s^d\tfrac{Q^2(\bar{0})}{Q^2(\bar{1})}+(1-s)^d}{(\tfrac{Q^2(\bar{0})}{Q^2(\bar{1})}+1)(s^d+(1-s)^d)}.\]
By Proposition \ref{pro: Eigenvalues}, the nontrivial eigenvalue of $S^2$ is given by $d\frac{\hat{Q}(\pi)}{\hat{Q}(0)}=d\frac{Q^2(\bar{0})-Q^2(\bar{1})}{Q^2(\bar{0})+Q^2(\bar{1})}$, so it is larger than one if and only if $\frac{Q^2(\bar{0})}{Q^2(\bar{1})}>\frac{d+1}{d-1}$.
	
\begin{figure}[h]
\centering
{\includegraphics[width=6.5cm]{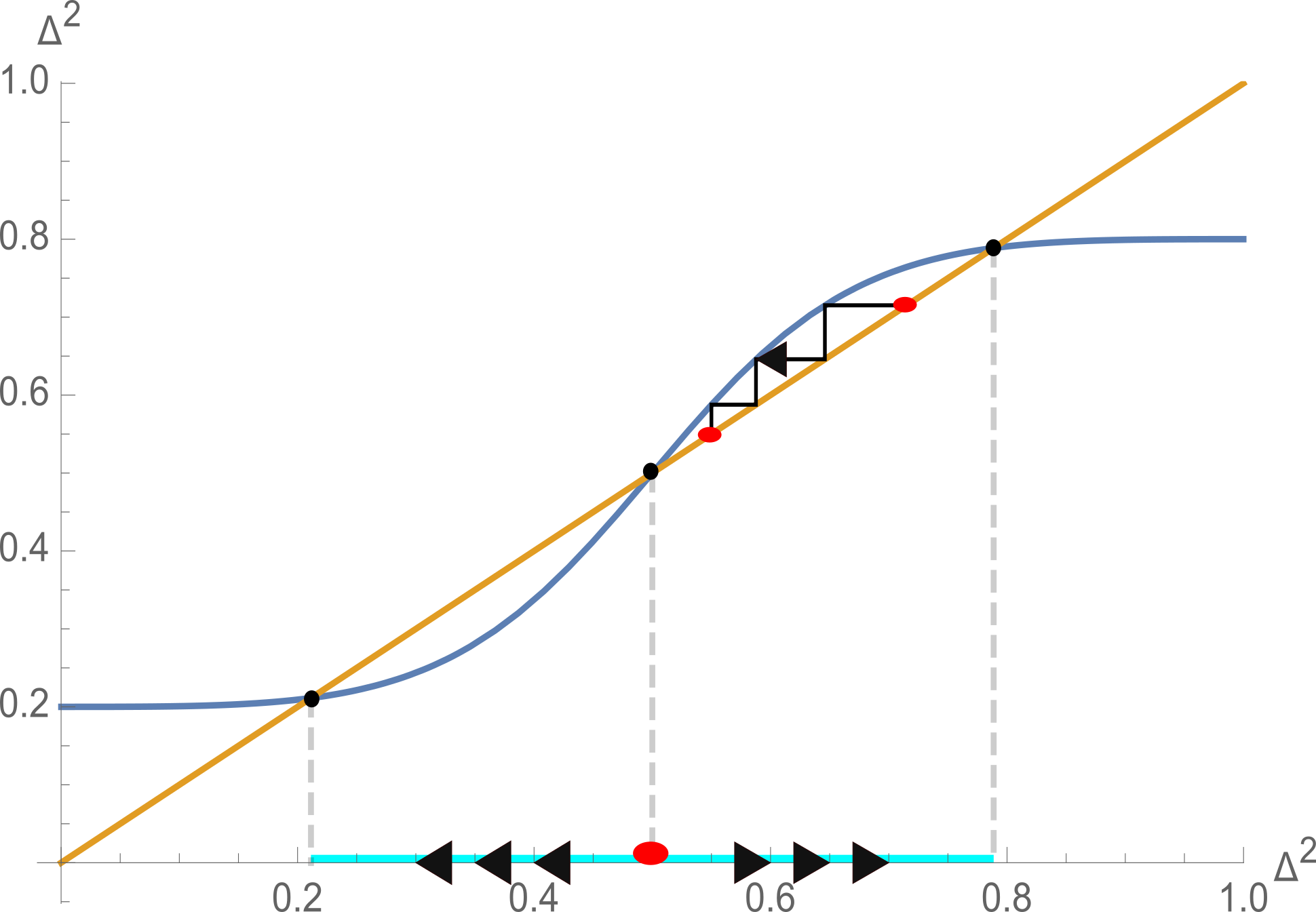}}
\caption{In the case $q=2$ the simplex is a unit interval. The picture shows the graph of the map $S_2$ where $d=3$ and $Q$ is such that $Q^2(\bar{0})/Q^2(\bar{1})=4>\frac{d+1}{d-1}$. The global unstable manifold around the equidistribution (center red point) is the cyan open interval between the two nontrivial fixed points of $S_2$. The solid line between the two graphs represents three backwards-iterations of the map $S_2$, starting at about $0.72$. }
\end{figure}
%\hyperref[pr: minPer]{Minimal periods in Theorem \ref{thm: main}}

%\sout{\begin{rk}
%We may compare this result to the different existence 
%Theorem $3$ in \cite{HK21} for automorphism invariant states different from the free state. 
%That theorem naturally works in a “low temperature regime" formulated in terms of certain $p$-norms of 
%$Q$, and provides non-uniqueness of automorphism invariant states for 
%large enough periods $q$. 
%Our present result is for states which are provably not invariant under {\color{blue} tree automorphisms, in particular translations of the tree}. It does 
%not need any such low temperature assumption and works for all summable $Q$.
%\end{rk}}
\newpage
Answering a question of a referee, we also state some result on nonexistence of non-t.i. states at fixed $q$ and sufficiently high temperatures. The proof is based on Dobrushin's uniqueness theorem.
\begin{pro}[]\label{pro: smallq}
Consider the $q$-spin model with fuzzy 
transfer operator $Q^q$ corresponding to the gradient potential $U$ and 
inverse temperature $\beta>0$.
\begin{itemize}
\item[i)] Define the $q$-variation of the gradient potential 
to be the number \[\delta_q(U):= \sup_{{1\leq a\leq q-1}\atop{ j\in \Z}}(U(|j+a|)-U(|j|)).\]
Suppose that $\beta (d+1)\delta_q(U)<2$. Then the $q$-spin model has a unique 
Gibbs measure. 
\item[ii)] Small $q$-uniqueness at sufficiently low $\beta$ holds for potentials $U$ for which there are $\alpha>0$ and $K>0$ such that $ U(j)\vert j \vert^{-\alpha} \stackrel{j \rightarrow \infty}{\rightarrow} K$ and which are additionally bounded from below. This class covers the SOS-model as well as the discrete Gaussian model ($\alpha=2$). \end{itemize}
\end{pro}
\begin{rk}
Note that ii) of Proposition \ref{pro: smallq}
extends i) to exponents $\alpha>1$ by a different proof, which does not give explicit bounds on $\beta$.
\end{rk}
In the following three subsections we assume $q \geq q_0(d)$ such that existence of $\tau$-unstable manifold $W_\tau$ of the non-linear 
map $S_q$ (see \eqref{eq: TheMapSq}) around the equidistribution is given.
\subsection{Boundary laws in forward and backwards direction}\label{subsec: recEq}
In this subsection we explicitly construct a radially symmetric  boundary law which is not translation invariant via backwards iteration.

Let $u=(u_{1}, \ldots,u_q)$ 
be any starting value chosen from the local unstable manifold $W_\tau$ of the non-linear 
map $S_q$ (see \eqref{eq: TheMapSq}) around the equidistribution. 
Fix any vertex $\rho \in V$  which we will refer to as the \textit{root}. 

First we define the boundary law values at edges $(x,y) \in \vec{L}^\rho$ pointing towards $\rho$, i.e. $d(\rho,x)=d(\rho,y)+1=n$ for some $n \in \mathbb{N}$ by setting \begin{equation} \label{eq: RecIn}
\lambda^{u}_{xy}=G_d(S_q^{1-n}(u))=\frac{(S_q^{1-n}(u))^{\odot d}}{\Vert (S_q^{1-n}(u))^{\odot d}\Vert_1} .
\end{equation} Here, $S_q^{1-n}$ denotes the $n-1$th iteration of the inverse $S_q^{-1}$.
By radial symmetry of the construction and definition of the function $S_q$, the boundary law equation is solved at any such edge $(x,y)$ pointing towards $\rho$. By construction they converge 
to the trivial fixed point one as the distance goes to infinity. 

For edges pointing away from the root the recursion formula for the boundary law reads as follows:
\begin{lemma}\label{lem: RecAway}
Consider an infinite path $\{\rho=x_0,x_1\},\{x_1,x_2\}, \ldots$ where $d(x_n,\rho)=n$. Then the recursion formula for the boundary law values at edges pointing away from the root reads as follows:
\begin{equation}\label{eq: RecAway} 
\begin{cases}
%\lambda^{\rho,u}_{x_{n+1}x_{n}}&=\frac{(S^{1-n}(u))^{\odot d}}{\Vert (S^{1-n}(u))^{\odot d}\Vert_1}, \ n \in \mathbb{N} \cup \{0\} \\
\lambda^{u}_{\rho x_1} &=\frac{S_q(u)^{\odot d}}{\Vert S_q(u)^{\odot d} \Vert_1 } \\
\lambda^{u}_{x_nx_{n+1}} &=\frac{Q^q(\lambda^{u}_{x_{n-1}x_n}) \odot (S_q^{1-n}(u_0))^{\odot d-1}}{\Vert Q^q(\lambda^{u}_{x_{n-1}x_n}) \odot (S_q^{1-n}(u_0))^{\odot d-1} \Vert_1}, \ n \in \mathbb{N}
\end{cases}
\end{equation}
\end{lemma}
%\hyperref[pr: lem4]{Go to proof}
\subsection{Convergence of boundary law values pointing away}\label{subsec: ConvNontrBL} 
By definition of the local unstable manifold the boundary law values \eqref{eq: RecIn} at edges pointing towards the root $\rho$ converge to the equidistribution as the distance to the root tends to infinity. The recursion formula \eqref{eq: RecAway} describing the boundary law values at edges pointing away from the root is more complicated. Nonetheless we will show that they also converge to the equidistribution as the distance to the root tends to infinity, a result which will then be employed to prove the lack of translation invariance of the associated gradient states stated in the Theorems \ref{thm: existence} and \ref{thm: main}.

Define 
$a_n:=S^{1-n}(u)^{\odot d-1}$. 
We know that for $u$ in $W_\tau$ we have 
$a_n\rightarrow eq^{\odot d-1}=C(q,d)eq$. 

It is convenient to discuss \eqref{eq: RecAway} in terms of the following map $F_a:\Delta^q\rightarrow \Delta^q$ 
where
\begin{equation*} 
\begin{split}
F_{a}(z)&=\frac{ Q^q z\odot a }{\Vert Q^q z\odot a \Vert_{1}}.
\end{split}
\end{equation*}
Then the convergence result reads
\begin{lemma}\label{lem: forewardConv} Suppose that $a_n$ are defined as above and $z:=\frac{S(u)^{\odot d}}{\Vert S(u)^{\odot d} \Vert_1}$, for $u \in W_\tau$ 
chosen in a sufficiently small neighborhood of the equidistribution $eq$ in $W_\tau$.
	
Then we have the convergence result
$F_{a_n}F_{a_{n-1}}\dots F_a(z) \stackrel{n \rightarrow \infty}{\rightarrow} eq$.
This means that the boundary law functions $\lambda^{u}_{x_nx_{n+1}}$ converge to the equidistribution along any infinite path pointing away from the root. 
\end{lemma}
%\hyperref[pr: lemforeward]{Go to proof}

\subsection{Transfer of lack of translation invariance}\label{Sec: non-translation}
In this subsection we will show that the gradient Gibbs measures to the boundary law constructed from a starting value $u \in W_\tau$ is not translation invariant. We show that uncountably many such $u$ exist.
\begin{pro}\label{pro: non-translation} Fix $q \geq q_0(d)$. Then there exist uncountably many starting values $u$ such that for any fixed $\rho \in V$ the gradient Gibbs measure $\nu^{\lambda^{u}}$ for the boundary law $\lambda^{u}$ with values given by the equations \eqref{eq: RecIn} and \eqref{eq: RecAway} is not translation invariant.
\end{pro}
We will now present the main ideas of the proof of Proposition \ref{pro: non-translation} in terms of the following two Lemmas.
To improve readability, we will use the following notation:
Let \[T_{\bar{j}}: \mathbb{R}^{\Z_q} \rightarrow  \mathbb{R}^{\Z_q} \, ; \, (T_{\bar{j}}w)(\cdot):=w(\cdot+\bar{j}) \] denote the cyclic shift of the index by $\bar{j} \in \Z_q$.

Further let $\langle \cdot, \cdot \rangle$ denote the Euclidean scalar product in $\mathbb{R}^q$.

We will first compare the marginals' distribution of the gradient Gibbs measure $\nu^{\lambda_q}$ along an edge starting from the root $\rho$ with the respective marginal along an edge at far distance from $\rho$. The convergence result for the boundary laws stated in Lemma \ref{lem: forewardConv} above then gives the following lemma.
\begin{lemma}\label{lem: originAndDistance}
Let $u \in W_\tau$ be any starting value for the recursion \eqref{eq: RecAway}. 
If there is some $\bar{j} \in \Z_q$ such that
\begin{equation} \label{eq: NotTheEquidistribution}
\langle Q^q(u^{\odot d})^{\odot d}, (T_{\bar{j}} u)^{\odot d}-u^{\odot d}\rangle \neq 0
\end{equation}
then the gradient Gibbs measure $\nu^{\lambda^u}$ is not translation invariant.
\end{lemma}
%\hyperref[pr: lem6]{Go to proof}
In the next step we may represent the ($\mathcal{C}^\infty$) local $\tau$-unstable manifold $W_\tau$ in a neighborhood of the equidistribution by the graph of a $\mathcal{C}^\infty$-function defined on a neighborhood of the equidistribution in the tangent space $\text{T}_{eq}W_\tau$.
For more details see also the proof of Thm.1.4.1 in \cite{Ch02} where existence of such a map is already shown to construct the local unstable manifold. 
%Hence, locally, each starting value $u$ for the recursion is given by the value of a $\mathcal{C}^\infty$-function at some tangent vector $v$. 
Performing a second-order Taylor expansion then yields the third statement of the following Lemma.
\begin{lemma}\label{lem: TaylorExp}
	\mbox{}\\ 
\begin{enumerate}
\item We may parametrize any element $u\in W_\tau$ near the equidistribution 
in terms of $v\in \text{T}_{eq} W_\tau$ 
in the corresponding stable linear space, in the form $u(v)=eq + v + h(v)$ 
with 
$h(v)\in \text{T}_{eq}\Delta^q$ in the orthogonal space to $\text{T}_{eq}W_\tau$ in $\text{T}_{eq}\Delta^q$, 
describing the deviation of the unstable manifold from its tangent space. 
\item By Theorem \ref{thm: Chaperon} we have $\text{T}_{eq} W_\tau=T^+_{eq} W_\tau \oplus T^-_{eq} W_\tau$ where \[
T^{+(-)}_{eq} W_\tau=\text{span}\{w \in \text{T}_{eq}\Delta^q \mid w \text{ is eigenvector of } \text{D}S_q \text{ to }\lambda>\tau \, (\lambda<-\tau)\}. \]
\item Assume that $T^{+,(-)}_{eq} W_\tau$ has positive dimension. Then there is an open neighborhood $V^{+(-)}$ of $0 \in T^{+,(-)}_{eq} W_\tau$ such that 
\begin{equation}\label{eq: TaylorExp}
\max_{\bar{j} \in \Z_q}
\vert \langle Q^q(u(v)^{\odot d})^{\odot d}, (T_{\bar{j}} u(v))^{\odot d} -u(v)^{\odot d}\rangle  \vert \geq  C \Vert v \Vert^2 
\end{equation} 
holds for all $v \in V^{+(-)}$.
\end{enumerate}
\end{lemma} 
%\hyperref[pr: lemTaylorExp]{Go to proof}
\subsection{Identifiability of the period $q$}
\begin{pro} \label{pro: Id}
Assume that $s$ and $t$ are coprime natural numbers. Then any of the $s$-height-periodic gradient Gibbs measures constructed in the Theorems \ref{thm: existence} and \ref{thm: main} is different from all of the $t$-height-periodic gradient Gibbs measures constructed in these Theorems.	
%Let $\lambda^t \in \mathbb{Z}_q^{\vec{L}}$ and $l^s \in \mathbb{Z}_s^{\vec{L}}$ be any height periodic boundary laws of \textbf{coprime} periods $t$ and $s$ which are radially symmetric with respect to the same vertex $\rho$.
%If for the associated gradient Gibbs measures we have $\nu^{\lambda^q}=\nu^{l^s}$ then $\nu^{\lambda^q}$ is the free state.
\end{pro}
%\hyperref[pr: proId]{Go to proof}
\begin{rk}
In particular, $q$-height-periodic GGMs indexed by distinct primes are distinct. 
\end{rk}
\section{Applications: The SOS-model and a heavy-tailed scenario} \label{sec: Applications}
In this section we will apply the general results on existence of translation non-invariant gradient Gibbs measures of general period $q$ stated in Section \ref{Sec: Existence theory for non-invariant GGMs} to two concrete examples. The first one is the well known SOS-model parametrized by the inverse temperature $\beta>0$. The second one, which we will refer to as inverse square model, is described by a transfer operator fixed to the value $1$ at zero and polynomially decaying of second order with linear dependence on a parameter $a>0$ away from zero.
Smaller values of the parameter $a$ amount to higher suppression of increments.

The inverse square  model serves as a manageable example in which the differential $\text{D}S_q[eq]$ can obtain both positive and negative eigenvalues. It is also the basis of Figure \ref{Fig: Nonhyperbolic} presented in Section \ref{Sec: Existence theory for non-invariant GGMs}.

Both models allow for explicit analysis, see Theorem \ref{thm: EB}.
This works particularly well in the two cases, 
as the Fourier transform of the transfer operator has good monotonicity properties, and 
in particular an expression for its pointwise inverse in 
terms of explicit functions. For any parametrized model with summable transfer 
operator $p\mapsto Q_{p}$ the same can be done in principle, if one provides 
the necessary additional (possibly numerical) 
input for the discussion of the pointwise inverse of its Fourier 
transform $k\mapsto \hat Q_{p}(k)$.
\subsection{The models}
%Recall that the Lemmas \ref{lem: Identifiability} and \ref{lem: TreeAutoGGM} give a criterion for existence of an uncountable family of non-translation invariant gradient Gibbs measures depending on the ratio of certain partial sums of the transfer operator. In this section we will illustrate our results by two concrete examples, the SOS-model and a logarithmic gradient potential.
\begin{table}[h]
\large
\begin{tabular}{l||l|l}
Model & SOS& Inverse square \\
\hline 
$Q(j)$ &$\exp(-\beta \vert j \vert)$&$\chi(j=0)+\chi(j \neq 0) \frac{a}{j^2}$ \\
\hline $\hat{Q}(k)$ &$\frac{e^{2 \beta}-1}{e^{2 \beta}-2e^{\beta}\cos k +1 }$& $1+\frac{a}{6}(3k^2-6\pi k+2\pi^2)$ \\
 \hline
$\frac{\hat{Q}(\pi)}{\hat{Q}(0)}$& $\tanh(\frac{\beta}{2})^2$ &$(1-a\frac{\pi^2}{6})/(1+a \frac{\pi^2}{3})$ \\ \hline 
$Q^q(\bar{i})$ & $\frac{\cosh\left(\beta(i-\frac{q}{2})\right)}{\sinh(\beta \frac{q}{2})}$ & $\frac{a}{q^2}\big(\zeta(2,\frac{i}{q})+\zeta(2,-\frac{i}{q})\big)$\\
& & $+1+\chi(\bar{i}\neq \bar{0})(\frac{aq^2}{i^2}-1)$   
\end{tabular}
\caption{The two models parametrized by $\beta$ ($a$, respectively), their Fourier-transforms and their mod-$q$ fuzzy operators. For $\bar i \in \Z_q$, the representative $i \in \Z$ occurring in the last row of the table is taken from $\{0, \ldots, q-1\}$. Here $\zeta(s,w)= \sum_{n=0}^\infty \frac{1}{(n+w)^s}$ denotes the Hurwitz zeta function which is accessible via numerical methods. It is defined for $w,s \in \mathbb{C}$ where $\vert s \vert >1$ and $0<\mathcal{R}(w) \leq 1$.
%\hyperref[pr: Fig2]{Calculations for Figure 2}	
}
\label{fig: theModels}
\end{table}

Note that although the gradient Gibbs measures constructed in Section \ref{Sec: Existence theory for non-invariant GGMs} are built from $q$-state clock models, the existence criterion presented in Theorem \ref{thm: existence} does not explicitly refer to the fuzzy transfer operators.
Nonetheless the fuzzy transfer operators are included in the synoptic Table \ref{fig: theModels} to give the reader an idea of their appearance and also highlight the benefits of presenting Theorem \ref{thm: existence} in terms of (the Fourier transform of) the original $\mathbb{Z}$-state transfer operator.

The spectra of $\text{D}S_q[eq]$ (up to the factor $d$) for particular choices of the parameters $\beta$ and $a$ are presented in Figure \ref{Fig: Spectra}.
\begin{center} %$\beta \in [1.2,3]$ $\beta \in [1.9,2.6]$
\begin{figure}[h]
\centering
\subfloat[SOS-model]
{\includegraphics[width=5cm]{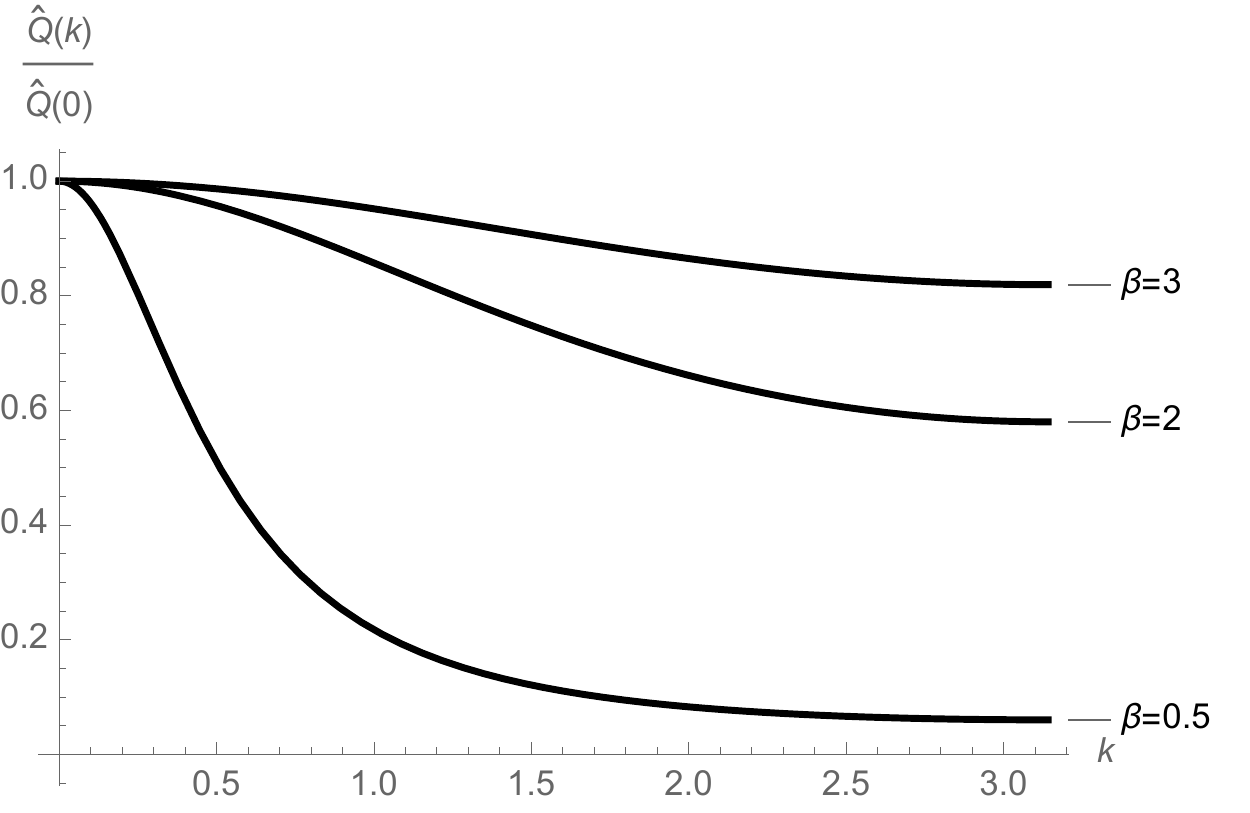}}
\quad
\subfloat[Inverse square-model ]{\includegraphics[width=5cm]{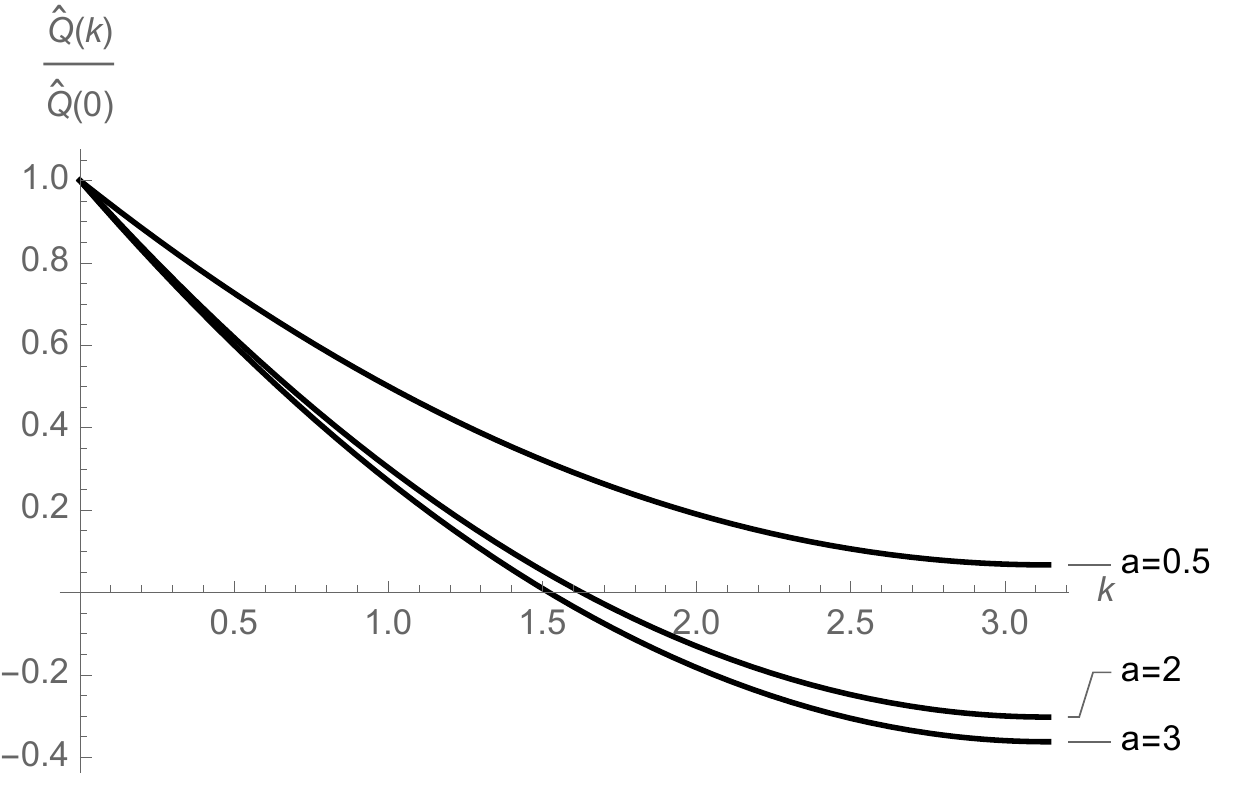}}
\caption{The graphs of the function $\hat{Q}(\cdot)/\hat{Q}(0)$ (see Theorem \ref{thm: existence}) for the two models at different parameter values.}
\label{Fig: Spectra}
\end{figure}
\end{center}
\begin{rk}
	Recall that by the Bochner-Herglotz representation theorem (eg. Thm 15.29 in \cite{Kl14}) the transfer operator $Q$ as a function on the integers is positive semidefinite (which by definition is equivalent to the positive semidefiniteness of all matrices of the form 
	$(Q(i_a-i_b))_{1\leq a,b\leq n}$ for all integers $n$ and all choices of $i_a\in \Z$)  
	if and only if it is the Fourier transform of a positive measure. In particular if $Q(0)< Q(i)$ for some $i\neq 0$ then $Q$ is not positive semidefinite and hence $\hat{Q}$ must obtain negative values in parts of its domain. The inverse square-model provides an example for this phenomenon.	
\end{rk}

\subsection{Explicit bounds}
Both models presented above have a decreasing Fourier transform at any particular choice of parameters $\beta$ and $a$. 
%Hence, if the existence criterion of Theorem \ref{thm: existence} is satisfied at $q=2$, it is also satisfied at any $q \geq 2$. 
The following theorem gives explicit and optimal
parameter regions where Theorem \ref{thm: existence} is applicable for all $q \geq 2$. Moreover, in the spirit of Theorem \ref{thm: main}, the minimal periods for the existence of gradient Gibbs measures which are not translation invariant are presented as functions on the whole parameter domain.    
\begin{thm}\label{thm: EB}
Consider the Cayley tree of order $d \geq 2$.
\begin{enumerate}
\item  We have the following parameter regions for which  $q$-height periodic gradient Gibbs measures which are not translation invariant exist for \textbf{all} periods $q\geq 2$.
\[\begin{cases}
\beta > \arcosh(\frac{d+1}{d-1}) \quad  &\text{for the SOS-model} \\
a \in (0, \infty) \setminus [\frac{6}{\pi^2}\frac{d-1}{d+2},\frac{9}{\pi^2}\frac{d+1}{d-3}]  \quad  &\text{Inverse square model}.
\end{cases}
\] 
\item Conversely, for $0<\beta\leq \arcosh(\frac{d+1}{d-1})$ and $a \in [\frac{6}{\pi^2}\frac{d-1}{d+2},\frac{9}{\pi^2}\frac{d+1}{d-3}]$, we have the following minimal periods such that for all $q$ greater or equal to them existence of $q$-height periodic gradient Gibbs measures which are not translation invariant is guaranteed.
\[
\begin{cases}
q_{\text{SOS}}(\beta,d)=\lceil \frac{2\pi}{\arccos(d-(d-1)\cosh(\beta))} \rceil  \quad  &\text{for the SOS-model} \\
q_{\text{InvSq}}(a,d)=\lceil \frac{2\pi}{\pi-\sqrt{\frac{\pi^2}{3}(1+\frac{2}{d})-\frac{2}{a}(1-\frac{1}{d})}} \rceil \quad  &\text{Inverse square model} \\
\end{cases}
\]	
Here $\lceil \cdot \rceil$ denotes the smallest integer bounding the nonnegative argument from above.
\item If $a \in (\frac{6}{\pi^2}\frac{d+1}{d-2},\frac{9}{\pi^2}\frac{d+1}{d-3}]$ then also $2$-height-periodic gradient Gibbs measures which are not translation invariant exist for the Inverse square model.
\end{enumerate}
\end{thm} 
%\hyperref[pr: thmEB]{Go to proof}

\begin{rk}
In the situation of the first statement of Theorem \ref{thm: EB} with $\beta > \arcosh(\frac{d+1}{d-1})$ and $a<\frac{6}{\pi^2}\frac{d-1}{d+2}$ at any $q \geq 2$ all eigenvalues of $\text{D}S_q[eq]$ are strictly greater than $1$ meaning that the unstable manifold $W_\tau(S_q) \subset \Delta^q$ has the full dimension $q-1$. Hence, the set of initial values in the construction of the gradient Gibbs measures lacking translation invariance has maximal degrees of freedom. 

By contrast, in the complement of these parameter regions, for large $q$ some eigenvalues will fail to be greater than one in modulus. Hence, the unstable manifold does not possess full dimensionality in that case.  
\end{rk}
\section{Proofs}\label{sec: Proofs}
\subsection{Proofs for Section \ref{sec: constr_GGM}}
\begin{proof}[Proof of Lemma \ref{lem: oldRepOfGGM}] \label{pr: lem1}
Let $\Lambda \subset V$ be any finite connected volume and $w \in \Lambda$ any fixed site. Then for any $\zeta_{\Lambda \cup \partial \Lambda} \in \mathbb{Z}^{\Lambda \cup \partial \Lambda}$ by the   equations \eqref{eq: ConstructionOfGGM} and \eqref{eq: RWPEtr} we have
\begin{equation} \label{eq: PrGGMoldRep1st}
\begin{split}
&\nu^{\lambda^q}(\eta_{\Lambda \cup \partial \Lambda}=\zeta_{\Lambda \cup \partial \Lambda})\cr 
&=T^q_Q(\mu^{\lambda^q})(\eta_{\Lambda \cup \partial \Lambda}=\zeta_{\Lambda \cup \partial \Lambda}) \cr &=\sum_{\bar{\omega}_{\Lambda \cup \partial \Lambda} \in \Z_q^{\Lambda \cup \partial \Lambda}}\mu^{\lambda^q}(\bar{\sigma}_{\Lambda \cup \partial \Lambda}=\bar{\omega}_{\Lambda \cup \partial \Lambda})\prod_{(x,y) \in {}^w\vec{L}, \, x,y \in \Lambda \cup \partial \Lambda} \rho^q_{Q_{\{x,y\}}}(\zeta_{(x,y)} \mid \bar{\omega}_y-\bar{\omega}_x) \cr
&=\sum_{\bar{\omega}_{\Lambda \cup \partial \Lambda} \in \Z_q^{\Lambda \cup \partial \Lambda}}\mu^{\lambda^q}(\bar{\sigma}_{\Lambda \cup \partial \Lambda}=\bar{\omega}_{\Lambda \cup \partial \Lambda})\prod_{(x,y) \in {}^w\vec{L}, \, x,y \in \Lambda \cup \partial \Lambda} \chi(\bar{\omega}_y-\bar{\omega}_x=\bar{\zeta}_{(x,y)})\frac{Q_{\{x,y\}}(\zeta_{(x,y)})}{Q^q_{\{x,y\}}(\bar{\zeta}_{(x,y)})}
\end{split}
\end{equation}	

As $\Lambda$ is connected, given the gradient configuration $\zeta_{\Lambda \cup \partial \Lambda} \in \mathbb{Z}^{\Lambda \cup \partial \Lambda}$ there is a one-to-one relation between the height $\bar{\omega}_w$ at site $w$ and the set of height configurations $\bar{\omega}_{\Lambda \cup \partial \Lambda} \in \Z_q^{\Lambda \cup \partial \Lambda}$ for which all indicators in \eqref{eq: PrGGMoldRep1st} are non-vanishing. 
Hence we can write the last expression in  \eqref{eq: PrGGMoldRep1st} as
\begin{equation}\label{eq: PrGGMoldRep2nd}
\begin{split}
&=\sum_{\bar{s} \in \Z_q}\mu^{\lambda^q}(\bar{\sigma}_w=\bar{s}, \bar{(\nabla \sigma)}_{\Lambda \cup \partial \Lambda}=\bar{\zeta}_{\Lambda \cup \partial \Lambda})\prod_{\{x,y\} \in L, \, x,y \in \Lambda \cup \partial \Lambda} \frac{Q_{\{x,y\}}(\zeta_{(x,y)})}{Q^q_{\{x,y\}}(\bar{\zeta}_{(x,y)})}.
\end{split}
\end{equation}	

Now, by Theorem \ref{thm: Zachary} applied to the finite-state-space measure $\mu^{\lambda^q}$, at any $\bar{s} \in \mathbb{Z}_q$ we have
\begin{equation}
\begin{split}
&\mu^{\lambda^q}(\bar{\sigma}_w=\bar{s}, \bar{(\nabla \sigma)}_{\Lambda \cup \partial \Lambda}=\bar{\zeta}_{\Lambda \cup \partial \Lambda})\cr &= Z_\Lambda^{-1}\prod_{y \in \partial \Lambda}\lambda^q_{yy_\Lambda}(\bar{s} +\sum_{b \in \Gamma(w,y_\Lambda)}\bar{\zeta}_b) \prod_{b \cap  \Lambda \neq \emptyset}Q^q_b(\bar{\zeta}_b),
\end{split}
\end{equation}
where $Z_\Lambda$ is a normalization constant.
Inserting this into \eqref{eq: PrGGMoldRep2nd} finishes the proof of Lemma \ref{lem: oldRepOfGGM}. 
\end{proof}
\begin{proof}[Proof of Lemma \ref{lem: BondMarginal}] \label{pr: lem2}
Apply Lemma \ref{lem: oldRepOfGGM} to $\Lambda=\{x\}$. Then summing over $\zeta_{(x,z)}$ for $z \in \partial \{x\} \setminus \{y\}$ gives
\begin{equation}
\begin{split}
\nu^{\lambda^q}(\eta_{(x,y)}=\zeta_{(x,y)})
&=Z_\Lambda^{-1}\sum_{\bar{s} \in \Z_q}\lambda^q_{yx}(\bar{s}+\bar{\zeta}_{(x,y)})Q_{\{x,y\}}(\zeta_{(x,y)}) \cr 
&\qquad \times  \left(\prod_{z \in \partial \{x\} \setminus \{y\}} \sum_{\zeta_{(x,z)}}\lambda^q_{zx}(\bar{s}+\bar{\zeta}_{(x,z)})Q_{\{x,z\}}(\zeta_{(x,z)}) \right) 
\end{split}
\end{equation}	
By the boundary law equation \eqref{eq: bl} the last expression in parentheses equals $\lambda^q_{xy}(\bar{s})$ up to a positive constant. This finishes the proof of Lemma \ref{lem: BondMarginal}. 
 
\end{proof}	
\begin{proof}[Proof of Theorem \ref{thm: DLR-eq}]\label{pr: thm2} \
By linearity of the DLR-equation \eqref{grad} and the map $T^q_Q$, and by extremal decomposition of Gibbs measures (e.g. Thm. 7.26 in \cite{Ge11}) it suffices to prove the statement for extremal Gibbs measures of $\gamma^q$. All of these are tree-indexed Markov chains (see Theorem 12.6 in \cite{Ge11}). Hence by Theorem \ref{thm: Zachary} it suffices to show that all measures of the form $\nu^{\lambda^q}$ as defined in \eqref{Not: BL_GGM} are invariant under the kernels \ref{Def: Gibbs specification}.
	
This means that we have to check $\nu^{\lambda^{q}}(A \mid \mathcal{T}^\nabla_\Lambda)=\gamma^\prime_\Lambda(A \mid \cdot)$ $\nu^{\lambda^{q}}$-a.s. for all $A \in \mathcal{F}^\nabla$ and all finite subtrees $\Lambda \subset V$. 
%As the local state space $\mathbb{Z}$ is countable it suffices to consider events of the form $A=\{\zeta \in \mathcal{F}^\nabla \mid \eta ()\zeta \}$, the $\sigma$-algebra generated by the gradients along edges with both vertices in $\Lambda \cup \partial \Lambda$.
Let $\Lambda \subset \Lambda \cup \partial \Lambda \subset \Delta$ be any finite subtrees and $\omega, \zeta \in \mathcal{F}^\nabla$ any gradient configurations with $\zeta_{V \setminus \Lambda} =\omega_{V \setminus \Lambda}$. To ease notation, in what follows we will omit the projection mappings and simply write $\nu^{\lambda^{q}}(\zeta_{\Lambda \cup \partial \Lambda})$ instead of $\nu^{\lambda^{q}}(\eta_{\Lambda \cup \partial \Lambda}=\zeta_{\Lambda \cup \partial \Lambda})$.
	
Then we have
\begin{equation}\label{eq: before_pinning}
\begin{split}
\frac{\nu^{\lambda^{q}}(\zeta_{\Lambda \cup \partial \Lambda} \mid [\omega]_{\partial \Lambda}, \omega_{(\Delta \cup \partial \Delta) \setminus \Lambda} )}{\nu^{\lambda^{q}}(\omega_{\Lambda \cup \partial \Lambda} \mid [\omega]_{\partial \Lambda}, \omega_{(\Delta \cup \partial \Delta) \setminus \Lambda} )}&= \frac{1_{[\zeta]_{\partial \Lambda}=[\omega]_{\partial \Lambda}}\nu^{\lambda^{q}}(\zeta_{\Lambda \cup \partial \Lambda}, \omega_{(\Delta \cup \partial \Delta) \setminus \Lambda} )}{\nu^{\lambda^{q}}(\omega_{\Lambda \cup \partial \Lambda}, \omega_{(\Delta \cup \partial \Delta) \setminus \Lambda})},
\end{split}
\end{equation}
where we used that $[\zeta]_{\partial \Lambda}$ is $\mathcal{F}^\nabla_{\Lambda \cup \partial \Lambda}$ measurable.
By Lemma \ref{lem: oldRepOfGGM} it follows
\begin{equation*}
\begin{split}
&\nu^{\lambda^{q}}(\zeta_{\Lambda \cup \partial \Lambda}, \omega_{(\Delta \cup \partial \Delta) \setminus \Lambda} )   
= \left( \sum_{\bar{s} \in \mathbb{Z}_q} \prod_{y \in \partial \Delta}\lambda^q_{y y_\Delta}(\bar{s} + \sum_{b \in \Gamma(v,y) }\zeta_b) \right)\prod_{b \cap \Delta \neq \emptyset}  Q_b(\zeta_b).
\end{split}
\end{equation*}
By the assumption $\zeta_{V \setminus \Lambda} =\omega_{V \setminus \Lambda}$ all factors in \eqref{eq: before_pinning} depending only on vertices in $\Lambda^c$ cancel: 
\begin{equation*}
\begin{split}
&\frac{1_{[\zeta]_{\partial \Lambda}= 
[\omega]_{\partial \Lambda}} \nu^{\lambda^{q}}(\zeta_{\Lambda \cup \partial \Lambda} \mid [\omega]_{\partial \Lambda}, \omega_{(\Delta \cup \partial \Delta) \setminus \Lambda} )}{\nu^{\lambda^{q}}(\omega_{\Lambda \cup \partial \Lambda} \mid [\omega]_{\partial \Lambda}, \omega_{(\Delta \cup \partial \Delta) \setminus \Lambda} )} \cr  	
&=\frac{1_{[\zeta]_{\partial \Lambda}= 
[\omega]_{\partial \Lambda}} \prod_{b \cap \Lambda \neq \emptyset} Q_b(\zeta_b) \sum_{\bar{s} \in \mathbb{Z}_q} \prod_{y \in \partial \Delta}\lambda^q_{y y_\Delta}(\bar{s} + \sum_{b \in \Gamma(v,y) }\zeta_b)}{\prod_{b \cap \Lambda \neq \emptyset} Q_b(\omega_b) \sum_{\bar{s} \in \mathbb{Z}_q} \prod_{y \in \partial \Delta}\lambda^q_{y y_\Delta}(\bar{s} + \sum_{b \in \Gamma(v,y) }\omega_b)}.
\end{split}
\end{equation*}
Now note that 
\begin{equation*}
\begin{split}
1_{[\zeta]_{\partial \Lambda}= 
[\omega]_{\partial \Lambda}} \sum_{\bar{s} \in \mathbb{Z}_q} \prod_{y \in \partial \Delta}\lambda^q_{y y_\Delta}(\bar{s} + \sum_{b \in \Gamma(v,y) }\zeta_b)=1_{[\zeta]_{\partial \Lambda}= 
[\omega]_{\partial \Lambda}} \sum_{\bar{s} \in \mathbb{Z}_q} \prod_{y \in \partial \Delta}\lambda^q_{y y_\Delta}(\bar{s} + \sum_{b \in \Gamma(v,y) }\omega_b)
\end{split},
\end{equation*}
as the outer summation is done over the same terms with shifted indices if the relative heights at $\partial \Lambda$ coincide.
This finally gives
\begin{equation*}
\begin{split}
\frac{\nu^{\lambda^{q}}(\zeta_{\Lambda \cup \partial \Lambda} \mid [\omega]_{\partial \Lambda}, \omega_{(\Delta \cup \partial \Delta) \setminus \Lambda} )}{\nu^{\lambda^{q}}(\omega_{\Lambda \cup \partial \Lambda} \mid [\omega]_{\partial \Lambda}, \omega_{(\Delta \cup \partial \Delta) \setminus \Lambda} )} 
&=\frac{1_{[\zeta]_{\partial \Lambda}= 
[\omega]_{\partial \Lambda}}  \prod_{b \cap \Lambda \neq \emptyset} Q_b(\zeta_b)}{ \prod_{b \cap \Lambda \neq \emptyset} Q_b(\omega_b)} \cr 
&=\frac{\gamma'_{\Lambda \cup \partial \Lambda}(\zeta_{\Lambda \cup \partial \Lambda} \mid \omega) }{\gamma'_{\Lambda  \cup \partial \Lambda}	(\omega_{\Lambda \cup \partial \Lambda} \mid \omega)}.
\end{split}
\end{equation*}
Summing over all possible $\zeta_{\Lambda \cup \partial \Lambda} \in \mathcal{F}^\nabla_{\Lambda \cup \partial \Lambda}$ and taking the limit $\Delta \rightarrow V$ (the specification $\gamma'$ is quasilocal) concludes the proof. 
%Note that it suffices to consider finite subtrees of the form $\Lambda \cup \partial \Lambda$ as the set of sets of this type is cofinal (see Remark 1.24 in \cite{Ge11}).
\end{proof}
\begin{proof}[Proof of Proposition \ref{pro: deloc}]\label{pr: pro1}
Consider a path $(b_1, b_2, \ldots, b_n)$ of length $n$ and let $\eta_b:=\sigma_x-\sigma_y$ denote the gradient spin variable along the edge $b=(x,y) \in \vec{L}$ and set $W_n:= \sum_{i=1}^n\eta_{b_i}$. 
Finally let $\bar{\sigma}=(\bar{\sigma}_{x_i})_{i=1, \ldots, n+1}$ be the fuzzy chain on $\Z_q$ along this path.
	
For any fixed $k \in \mathbb{Z}$ we have
\begin{equation*}
\nu^{\lambda^q}(W_n=k) \, = \, \int \mu^{\lambda^q}(\text{d}\bar{\sigma}) \, \nu^{\lambda^q}(W_n=k \mid \bar{\sigma}).
\end{equation*}  
In the rest of the proof we will show that $\nu^{\lambda^q}(W_n=k \mid \bar{\sigma}) \stackrel{n \rightarrow \infty}{\rightarrow} 0$ uniformly in $\bar{\sigma}$. To start with, first recall that $\nu^{\lambda^q}(W_n=k \mid \bar{\sigma})$ is just the product measure of the measures $\rho_q^Q$ describing the marginals along the edges of the path conditioned on the increment of the fuzzy chain along the respective edge. Given the increment of the fuzzy chain along an edge, the measure $\rho_q^Q$ does not have any spatial dependence.
Hence the remainder of the proof is a direct generalization of the proof of Theorem 4 in \cite{HK21} which states delocalization of $\nu^{\lambda^q}$ in the special case of a spatially homogeneous boundary law $\lambda^q$.
\end{proof}
\subsection{Proofs for Section \ref{Sec: Existence theory for non-invariant GGMs}}
\begin{proof}[Proof of Proposition \ref{pro: Eigenvalues} ] \label{pr: pro2}
The proof consists of first proving that \begin{equation} \label{eq: DSandQq}
\text{D}S_q[eq]=\frac{d}{\Vert Q \Vert_1}Q^q
\end{equation} and afterwards expressing the spectrum of $Q^q$ by the Fourier-transform of $Q$.
	
For any $v \in \Delta^q$ we have $S_q(v)=\frac{Q^qv^{\odot d}}{\Vert Q^q \Vert_1 \Vert v^{\odot d}\Vert_1}=\frac{1}{\Vert Q^q \Vert_1}Q^q(G_d(v))$. Direct calculation gives the following
\begin{lemma}
For any $s \neq 0$ the map $G_s: \Delta^q \rightarrow \Delta^q \, ; \, z \mapsto \frac{z^{\odot s}}{\Vert z^{\odot s} \Vert_1}$ is a $\mathcal{C}^\infty$-diffeomorphism with inverse $G_{\frac{1}{s}}$. The differential at $z \in \Delta^q$ is given by
\begin{equation*}
{\text{D}  G_s}[z](v)=\frac{s}{\Vert z^{\odot s}\Vert_1}\left(z^{\odot s-1} \odot v-\frac{z^{\odot s} \langle  z^{\odot s-1},v \rangle}{\Vert z^{\odot s}\Vert_1}\right)
\end{equation*}
for any $v \in T_z\Delta^q$.
\end{lemma}	
Hence, applying the chain rule we arrive at
\begin{equation*}
\begin{split}
\text{D}S_q[eq](v) &=\frac{1}{\Vert Q^q \Vert_1}Q^q(\text{D}G_d[eq](v)) \cr&=\frac{1}{\Vert Q^q \Vert_1}dq^{d-1}Q^q(eq^{\odot d-1} \odot v)) \cr 
&=\frac{d}{\Vert Q^q \Vert_1}Q^q(v) \cr 
&=\frac{d}{\Vert Q \Vert_1}Q^q(v) 
\end{split}
\end{equation*}
where the second term of $\text{D}G_d[eq](v)$ vanishes as $\langle eq^{\odot{d-1}},v\rangle=0$.

The eigenvalues of the $q$-circulant matrix $Q^q$
are given by Fourier transform
\begin{equation}
\begin{split}
\lambda_j(Q^q)= \sum_{r=0}^{q-1}Q^q(r)e^{i \frac{2 \pi j r}{q}}
\end{split}
\end{equation}
where $j=0,1,\dots, q-1$. 
The corresponding
eigenvectors are $\frac{1}{\sqrt{q}}(1,\omega_j,\dots,\omega_j^{q-1})$
with $\omega_j=e^{i \frac{2 \pi j}{q}}$. 

By symmetry of $Q^q$ we are reduced to the real cosine-Fourier modes and only $1+\lfloor \frac{q}{2} \rfloor$ distinct eigenvalues for which we can choose an orthogonal basis of eigenvectors in $\mathbb{R}^q$. These are given by suitable linear combinations of the real and imaginary parts of the complex eigenvectors stated above. 
%A possible choice is given by the vectors 
%\[(\ 1,\Re(\omega_j)+\Im(\omega_j), \ldots, \Re(\omega_j^{q-1})+\Im(\omega_j^{q-1}) \ )\] where $j=0, \ldots, q-1$.
Inserting the definition of the fuzzy operator function in terms of sums 
of elements of equivalence classes we obtain
\begin{equation}
\begin{split}
\lambda_j(Q^q)&= Q^q(0)+\sum_{r=1}^{q-1}Q^q(r)\cos\left(\frac{2 \pi j r}{q} \right)\cr
&=\sum_{n\in \Z}Q(n)\cos\bigl(2 \pi n \frac{j}{q}\bigr)
\end{split}
\end{equation}
where $j=0,1,\dots,  \lfloor \frac{q}{2} \rfloor$. 

Hence, we see that the eigenvalues of the $q$-fuzzy operator 
are given by the Fourier transform of the original transfer operator 
\begin{equation} \label{eq: Fourier}[-\pi,\pi)\ni k \mapsto \hat Q(k)=\sum_{n\in \Z}Q(n)\cos(n k ) \end{equation}
sampled at the finitely many values $2 \pi\frac{j}{q}, \, j=0,1,\dots,  \lfloor \frac{q}{2} \rfloor$. 

In the last step we insert $\Vert Q \Vert_1=\hat{Q}(0)$ into \eqref{eq: DSandQq} finishing the proof of Proposition \ref{pro: Eigenvalues}.
\end{proof}
\begin{proof}[Proof of Proposition \ref{pro: smallq}]
To prove part i) note that
	Proposition 8.8 in \cite{Ge11} derived from Dobrushin uniqueness theory, and 
	specialized to our situation of pair interactions of a regular tree 
	ensures uniqueness  under 
	the condition $(d+1)\delta (\log Q^q)<2$, where 
	$\delta (\log Q^q) :=\max_{0\leq a,b\leq q-1}(\log Q^q(a)-\log Q^q(b))$ is the variation 
	of the logarithm of the fuzzy transfer operator. 
	
	To relate this condition back to our original interaction potential $U$ let us write 
	\begin{equation}\begin{split}
	&\delta (\log Q^q)=\max_{0\leq a,b\leq q-1}\Bigl(
	\log \sum_{j\in \Z}\exp\bigl(-\beta U(|a+ q j|)\bigr)- \log \sum_{j\in \Z}\exp\bigl(-\beta U(|b+ q j|)\bigr)
	\Bigr) \cr
	&\leq
	\max_{0\leq a,b\leq q-1}\log \frac{\sum_{j\in \Z}\exp\bigl(-\beta (U(|a+ q j|)-U(|b+ q j|)\bigr)   \exp\bigl(-\beta U(|b+ q j|\bigr) }
	{\sum_{j\in \Z}\exp\bigl(-\beta U(|b+ q j|)\bigr)}
	\cr
	&\leq
	\beta \delta_q(U),
	\end{split}
	\end{equation}
	where the last estimate follows as we have for all terms appearing in 
	the exponential under the expectation the uniform bound 
	$| U(|a+ q j|)-U(|b+ q j|) | \leq \delta_q(U)$. 
To prove part ii) we assume that $U(x)\geq -M$ with $M\geq 0$ finite.
We will use an integral approximation to show that $\delta(\log Q^q) \stackrel{\beta \downarrow 0}{\rightarrow}0$, i.e.,	
\begin{equation}\begin{split}
&\lim_{\beta\downarrow 0}\max_{0\leq a\leq q-1}\Bigl |
\log \sum_{j\in \Z}\exp\bigl(-\beta U(|a+ q j|)\bigr)- \log \sum_{j\in \Z}\exp\bigl(-\beta 
U(|q j|)\bigr)
\Bigr |=0. \cr
\end{split}
\end{equation}
It suffices to assume $K=1$, by rescaling of $\beta$. 	
Put $\gamma(\beta)=\beta^\frac{1}{\alpha}$ and write the term under the modulus in the last display 
as 
\begin{equation}
\begin{split}
&\log \frac{
	\sum_{j\in \Z}\exp\bigl(-\gamma^\alpha U(|a+ q j|)\bigr)\gamma}
{\sum_{j\in \Z}\exp\bigl(-\gamma^\alpha U(|a+ q j|)\bigr)\gamma
}=\log\frac{\int_{\mathbb{R}}g_{a,\gamma} d\lambda}{\int_{\mathbb{R}}g_{0,\gamma} d\lambda}
\end{split}
\end{equation}
with  the function	
\[g_{a,\gamma}(x)=\sum_{j\in \Z}\exp\Bigl(-\gamma^\alpha U(|a+ q j| )\Bigr)1_{[
	\gamma j,\gamma (j+1))}(x).\]
We would like to perform the limit $\gamma\downarrow 0$ and conclude 
that numerator and denominator under the logarithm converge 
towards the same finite limit. 	
In the first step we show that for Lebesgue-a.e. $x$ the pointwise limit holds
\[\lim_{\gamma\downarrow 0}
g_{a,\gamma}(x)=e^{-q^\alpha |x|^\alpha},\] where it is important 
that the limiting function is independent of $a$. 
To do so, fix $x\neq 0$ and let $j(x,\gamma)\in \Z$ be given by 
$x\in \gamma [j(x,\gamma), j(x,\gamma)+1)$. 
Then $\lim_{\gamma \downarrow 0}\gamma j(x,\gamma)=x$ and we conclude, for fixed $x\neq 0$, the desired pointwise limit
\begin{equation}\label{eq: Fixq}
\begin{split}
&\gamma^\alpha U(|a+ q j(x,\gamma)| )\cr
&	=\Bigr(\gamma |a+ q j(x,\gamma)| \Bigl)^\alpha\,\,\frac{U(|a+ q j(x,\gamma)| )}{ |a+ q j(x,\gamma)|^\alpha}\rightarrow q^\alpha |x|^\alpha \times 1 
\end{split}
\end{equation}	
with $\gamma \downarrow 0$, by our main assumption on the 
growth of $U$. 

Next, convergence of the integrals under the logarithm follows and the proof is finished, 
once we can construct 
a Lebesgue-integrable dominating function for small enough $\gamma$.
As $a \leq q$ the expression $a+ q j(x,\gamma)$ is non-negative if and only if $j(x,\gamma) \geq 0$. Hence an inspection of the first factor of \eqref{eq: Fixq} gives the lower bound
\[\gamma \vert a+qj(x,\gamma) \vert \geq q\vert x \vert +\gamma \min(a-q,-a) \geq q (\vert x \vert-\gamma).  \]
%Inserting the lower bound into \eqref{eq: Fixq}, we obtain
%\[ \vert \ \Bigr(\delta |a+ q j(x,\delta)| \Bigl)^\alpha-q^\alpha \vert x \vert^\alpha \ \vert\leq q^\alpha (\vert x \vert +\delta)^\alpha-q^\alpha\vert x \vert^\alpha\]
%\ \vert (1-\delta \frac{1+a}{q \vert x \vert})^\alpha -1 \vert \]
Combining this lower bound with the assumption on the growth rate of $U$, we conclude
that there are $R\in(0,\infty)$ and $\gamma_0>0$ such that for all $\gamma<\gamma_0$ and all $|x|\geq R$
we have $g_{a,\gamma}(x)\leq e^{-\frac{q^\alpha |x|^\alpha}{2}}$. 
For $\vert x \vert <R$ we take the lower bound $M$ on $U$ into account.
Hence, the function $e^{-\frac{q^\alpha |x|^\alpha}{2}}1_{|x|\geq R}+ e^{\gamma^\alpha_0 M}1_{|x|\leq R}$ 
can be used as Lebesgue-integrable dominating function for small enough $\gamma<\gamma_0$ 
and this finishes the proof. 
%	To prove small-$q$-uniqueness for polynomially growing potentials $U(i)=\vert i \vert^\alpha$ in the superlinear case $\alpha>1$, we use an integral approximation to show that $\delta(\log Q^q) \stackrel{\beta \downarrow 0}{\rightarrow}0$, i.e.,	
%	\begin{equation}\begin{split}
%	&\lim_{\beta\downarrow 0}\max_{0\leq a\leq q-1}\Bigl |
%	\log \sum_{j\in \Z}\exp\bigl(-\beta |a+ q j|^\alpha\bigr)- \log \sum_{j\in \Z}\exp\bigl(-\beta |q j|^\alpha\bigr)
%	\Bigr |=0. \cr
%	\end{split}
%	\end{equation}
%	Put $\gamma(\beta)=a \beta^\frac{1}{\alpha}$, $\delta(\beta)=q\beta^\frac{1}{\alpha}$ and write the term under the modulus 
%	as 
%	\begin{equation}
%	\begin{split}
%	&\log \frac{
%		\sum_{j\in \Z}\exp\bigl(-|\gamma+ \delta j|^\alpha)\bigr)\delta}
%	{\sum_{j\in \Z}\exp\bigl(- |\delta j|^\alpha)\bigr)\delta}=\log\frac{\int_{\mathbb{R}}f_{\gamma,\delta} d\lambda}{\int_{\mathbb{R}}f_{0,\delta} d\lambda}
%	\end{split}
%	\end{equation}
%	with $f_{\gamma,\delta}(x)=\sum_{j\in \Z}\exp\bigl(-|\gamma+ \delta j|^\alpha)1_{[
%		\delta j,\delta (j+1))}(x)$. We have for Lebesgue-a.e. $x$ that $\lim_{\gamma,\delta\downarrow 0}
%	f_{\gamma,\delta}(x)=e^{-|x|^\gamma}$ with the Lebesgue-integrable dominating function $2 e^{-|x|^\alpha}$. 
%	By dominated convergence this implies convergence of numerator and denominator against the same 
%	finite integral as $\gamma,\delta\downarrow 0$ which proves the claim. 
\end{proof}	
\begin{proof}[Proof of Lemma \ref{lem: RecAway}]\label{pr: lem4}
Consider an infinite path $\{\rho=x_0,x_1\},\{x_1,x_2\} \dots$.

For any $n \in \mathbb{N}$ let $y_1, \ldots y_d \in \partial \{x_n\} \setminus \{x_{n+1}\}$ such that $y_d=x_{n-1}$. Then the set of oriented edges of the form $(y_i,x_n)$, $i=1, \ldots d$ splits up into the edge $(x_{n-1}, x_n)$ pointing away from the root and the $d-1$ edges $(y_i,x_n)$, $i=1, \ldots d-1$ pointing towards the root where $\text{d}(y_i,\rho)=\text{d}(x_n,\rho)+1=n+1$. From radial symmetry of the construction and the boundary law equation \eqref{eq: perGenBL} it thus follows that
\begin{equation*} 
\begin{split}
\lambda^{u}_{x_nx_{n+1}}&=\frac{Q^q(\lambda_{x_{n-1}x_n}) \odot (Q^q(\lambda_{y_1x_n}))^{\odot d-1}}{\Vert Q^q(\lambda_{x_{n-1}x_n}) \odot (Q^q(\lambda_{y_1x_n}))^{ \odot d-1} \Vert_1} \cr 
&=\frac{Q^q(\lambda_{x_{n-1}x_n})\odot S_q(G_{\frac{1}{d}}(\lambda_{y_1x_n}))^{\odot d-1}}{\Vert Q^q(\lambda_{x_{n-1}x_n}) \odot S_q(G_{\frac{1}{d}}(\lambda_{y_1x_n}))^{\odot d-1} \Vert_1}  \cr
&=\frac{Q^q(\lambda_{x_{n-1}x_n}) \odot (S_q(S_q^{-n}(u_0)))^{\odot d-1}}{\Vert Q^q(\lambda_{x_{n-1}x_n}) \odot (S_q(S_q^{-n}(u_0)))^{\odot d-1} \Vert_1}  \cr
&=\frac{Q^q(\lambda_{x_{n-1}x_n}) \odot (S_q^{1-n}(u_0))^{\odot d-1}}{\Vert Q^q(\lambda_{x_{n-1}x_n}) \odot (S_q^{1-n}(u_0))^{\odot d-1} \Vert_1}
\end{split}
\end{equation*}
where as above the $\odot$-operator denotes the Hadamard product.

In the case $n=0$, i.e. for the edge $(\rho, x_1)$, we simply obtain \[\lambda^u_{\rho x_1}=G_d(S_q(u))=\frac{S_q(u)^{\odot d}}{\Vert S_q(u)^{\odot d} \Vert_1 } \]
as all $d$ edges of the form $(y_i,\rho)$, $y_1, \ldots, y_d \in \partial \{ \rho \} \setminus \{x_1\}$ are pointing towards the root.
\end{proof}
\begin{proof}[Proof of Lemma \ref{lem: forewardConv}] \label{pr: lemforeward}
It is convenient for the proof 
to work with distance on $\Delta^q$ provided 
by the restriction of the Euclidean norm. 
	
The proof rests on the following local uniform contraction estimate. 
\begin{lemma} \label{lem: unifContr}There is a positive 
contraction coefficient $\gamma<1$, and 
there are neighborhoods $U,W $ of the equidistribution such 
that  for all $a\in U$ and for all $w,w'\in W_\tau$ the following locally uniform 
contraction holds in two-norm. 
$$\Vert F_{a}(w)-F_a(w')\Vert_2 \leq \gamma \Vert w-w'\Vert_2
$$
\end{lemma}
\begin{proof}[Proof of Lemma \ref{lem: unifContr}]
Choose $\gamma$ strictly between 
$M:=\max_{j=1,\dots, q-1}|\lambda_j|$ and $1$, where $\lambda_j$ 
denote the eigenvalues of the symmetric normalized 
transfer operator $\frac{Q^q}{\Vert Q^q\Vert_1}$
acting on the tangent space $\text{T}_{eq}\Delta^q$.  
We have for $z\mapsto F_{eq}(z)=
\frac{Q^q z}{\Vert Q^q z\Vert_{1}}$ that the 
$k$-th component of the differential $\text{D}F_{eq}[eq]$ 
taken in $z=eq$ is given by   
\begin{equation*} 
\begin{split}
\text{D}F_{eq}[eq]_k&=\frac{ Q^q_k}{\Vert Q^q\Vert_{1}} 
\end{split}
\end{equation*}
where the denominator is 
the one-norm of the function $Q^q$. 
Hence $\text{D} F_{eq}[eq]|_{T \Delta^q}=
\frac{Q^q}{\Vert Q^q\Vert_1}$. 
Hence we know that the $l^2$-operator norm of 
$\text{D}F_{eq}[eq]|_{\text{T}_{eq}\Delta^q}$ is equal to $M<1$. 
But as the function $(z,a)\mapsto \text{D}F_{a}[z]$ is jointly continuous, we may 
perturb this estimate and see the following. 
For each $\gamma$ strictly bigger than $M$,  
we may find neighborhoods $U,W $ of the equidistribution such 
that  for all $a\in U$ and for all $z\in W$ the differential satisfies 
the uniform norm-estimate
\[\Vert \text{D} F_{a}[z]\Vert_{Op,l^2(\text{T}_{eq} \Delta^q)} \leq \gamma 
\] But from this follows the desired Lipschitz property for 
the two-norm on $\text{T}_{eq}\Delta^q$. This concludes the proof of the lemma. 
\end{proof}
We continue with the proof of Lemma \ref{lem: forewardConv} and split the relevant quantity 
in two terms 
\begin{equation} \label{eq: splitUC}
\begin{split}
\Vert F_{a_n}F_{a_{n-1}}\dots F_a(z)-eq\Vert_2
&\leq \Vert F_{a_n}F_{a_{n-1}}\dots F_a(z)-  
F_{a_n}F_{a_{n-1}}\dots F_a(eq)\Vert_2 \cr
& \quad +\Vert F_{a_n}F_{a_{n-1}}\dots F_a(eq)-eq \Vert_2
\end{split}
\end{equation}
To be able to apply the local uniform contraction property of Lemma \ref{lem: unifContr}, we first have to ensure that $u$ can be chosen such that for all $n \in \mathbb{N}$ both $F_{a_n}F_{a_{n-1}}\dots F_a(eq)$ and $F_{a_n}F_{a_{n-1}}\dots F_a(z)$ stay sufficiently close to the equidistribution. Afterwards we show that both terms converge to zero as $n$ tends to infinity.
We start with estimating the second term of \eqref{eq: splitUC}. 
	
Let $\varepsilon>0$. We will show by induction on $n \in \mathbb{N}$ that $\Vert F_{a_n}F_{a_{n-1}}\dots F_a(eq)-eq \Vert < \varepsilon$ for all $n \in \mathbb{N}_0$.  As $(a_n)_{n \in \mathbb{N}}$ converges exponentially fast to the equidistribution (since $a_1 \in W_\tau$), continuity of $F$ in $a$ allows to take the starting value $a=a_1$ such that $b_n:=\Vert F_{a_n}(eq) - eq\Vert< (1-\gamma)\varepsilon$ for all $n \in \mathbb{N}$. In particular, this already proves the initial step. The induction step then follows from
\begin{equation}\label{eq: induction1}
\begin{split}
&\Vert F_{a_n}F_{a_{n-1}}\dots F_a(eq)-eq \Vert_2\cr
&\leq \Vert F_{a_n}F_{a_{n-1}}\dots F_a(eq)-F_{a_n}(eq) \Vert_2
+\Vert F_{a_n}(eq) - eq\Vert_2 \cr
&\leq \gamma \Vert F_{a_{n-1}}\dots F_a(eq)-eq \Vert_2
+b_n,
\end{split}
\end{equation}
where in the second inequality we used the induction hypothesis for $n-1$ to apply Lemma \ref{lem: unifContr}.
	
Repeated application of \eqref{eq: induction1} then gives the following estimate on the second term of \eqref{eq: splitUC}:
\begin{equation} 
\begin{split}
&\Vert F_{a_n}F_{a_{n-1}}\dots F_a(eq)-eq \Vert_2 \cr
&\leq \gamma \Vert F_{a_{n-1}}\dots F_a(eq)-eq \Vert_2+b_n \cr 
&\leq \dots \leq \sum_{j=0}^n
\gamma^{n-j} b_j. 
\end{split}
\end{equation} 
Fix $\tilde{\varepsilon}>0$. As $b_n$ converges to zero as $n$ tends to infinity, there is an $m$ such that $b_j\leq \tilde{\varepsilon}$ for all 
$j\geq m$. Then 
\begin{equation} 
\begin{split}
&\sum_{j=0}^n
\gamma^{n-j} b_j =\sum_{j=0}^{m-1}
\gamma^{n-j} b_j + \sum_{j=m}^n
\gamma^{n-j} b_j \cr
&\leq \frac{\gamma^{n-m+1}}{1-\gamma}\sup_{j\geq 0}b_j
+\frac{\tilde{\varepsilon}}{1-\gamma}.
\end{split}
\end{equation} 
Hence, considering $\tilde{\varepsilon} \downarrow 0$, we have $\Vert F_{a_n}F_{a_{n-1}}\dots F_a(eq)-eq \Vert \stackrel{n \rightarrow \infty}{\rightarrow} 0$.
	
It remains to estimate the first term of \eqref{eq: splitUC}.
By the local uniform contraction property of the 
Lemma \ref{lem: unifContr} we have
\begin{equation} \label{eq: ApplUC}
\begin{split}
&\Vert F_{a_n}F_{a_{n-1}}\dots F_a(z)-  
F_{a_n}F_{a_{n-1}}\dots F_a(eq)\Vert_2 \cr
&\leq \gamma \Vert F_{a_{n-1}}\dots F_a(z)-  
F_{a_{n-1}}\dots F_a(eq)\Vert_2 \cr
&\leq \dots \leq \gamma^{n}\Vert z -eq\Vert_2 \stackrel{n \rightarrow \infty}{\rightarrow} 0.
\end{split}
\end{equation}
Here, the condition that $F_{a_n}F_{a_{n-1}}\dots F_a(z)$ stays sufficiently close to the equidistribution again follows from induction on $n$ 
%using the estimates \eqref{eq: splitUC} and \eqref{eq: ApplUC}
provided that $z$ is taken sufficiently close to the equidistribution. This concludes the proof of Lemma \ref{lem: forewardConv}.  
\end{proof}
\begin{proof}[Proof of Lemma \ref{lem: originAndDistance}] \label{pr: lem6}
Consider an infinite path $\{\rho=x_0,x_1\},\{x_1,x_2\}, \ldots$ where $d(x_n,\rho)=n$. Further assume that in the group representation $x_n=g_1\ldots g_n$ we have $g_1 \neq g_m$ for all $1 < m \leq n$. 
For any $n \in \mathbb{N}$ define a translation $\Theta_{n}: V \rightarrow V \ ; \ \Theta_n(\cdot):=g_1 \ldots g_n \ \cdot $ where $g_1 \ldots g_n$ is the group representation of the vertex $x_n$. Then, for all $m \in \mathbb{N}$, we have $d(\rho,\Theta_{n}(x_m))=n+m$. 
Employing the fact that by Lemma \ref{lem: forewardConv} the boundary law functions $\lambda^u_{x_nx_{n+1}}$ and $\lambda^{\rho, u}_{x_{n+1},x_n}$ converge to the equidistribution as $n$ tends to infinity, we will compare the marginals' representation of the measure $\nu^{\lambda^q}$ along the edge $\{\rho,x_1\}$ to that along the edge $\{x_nx_{n+1}\}=\{\Theta_n(\rho),\Theta_{n}(x_1)\}$ for large $n$. The statement of the lemma then follows by showing that \eqref{eq: NotTheEquidistribution} implies that the marginals distribution along the edge $\{\rho,x_1\}$ is different from the equidistribution.
Take any $j_1,j_2 \in \mathbb{Z}$. Then translation invariance of the measure $\nu^{\lambda^u}$  would imply that
\begin{equation}\label{eq: NonhomQuot}
\frac{\nu^{\lambda^{u}}(\eta_{(\rho,x_1)}=j_1)}{\nu^{\lambda^u}(\eta_{(\rho,x_1)}=j_2)}=\frac{\nu^{\lambda^u}(\eta_{(x_n,x_{n+1})}=j_1)}{\nu^{\lambda^u}(\eta_{(x_n,x_{n+1})}=j_2)}
\end{equation} 
Inserting the statement of Lemma \ref{lem: BondMarginal} for the marginal along an edge 
\begin{equation*}
\begin{split}
\nu^{\lambda^u}(\eta_{(x_n,x_{n+1})}=j)&=Z^{-1}_{(x_n,x_{n+1})}Q(j) \sum_{\bar{i} \in \Z_q} \lambda^u_{x_n x_{n+1}}(\bar{i}+\bar{j})\lambda^u_{x_{n+1}x_n}(\bar{i}) \cr 
&=Z^{-1}_{(x_n,x_{n+1})}Q(j) \langle \lambda^u_{x_nx_{n+1}}, T_{\bar{j}}(\lambda^u_{x_{n+1}x_n}) \rangle
\end{split}
\end{equation*}
into \eqref{eq: NonhomQuot} and letting $n$ tend to infinity where the r.h.s converges to $1$, we arrive at the necessary criterion
\begin{equation*}
\frac{\langle \lambda^u_{\rho,x_1},T_{\bar{j}_1} \lambda^u_{x_1,\rho} \rangle}{\langle \lambda^u_{\rho,x_1},T_{\bar{j}_2} \lambda^u_{x_1,\rho}\rangle}=1.
\end{equation*}
Hence, inserting $\lambda^u_{x_1\rho}=\frac{u^{\odot d}}{\Vert u^{\odot d} \Vert_1}$ and  $\lambda^u_{\rho x_1}=\frac{S(u)^{\odot d}}{\Vert S(u)^{\odot d} \Vert_1}=\frac{Q^q(u^{\odot d})}{\Vert Q^q(u^{\odot d}) \Vert_1 }$  
the measure $\nu^{\lambda^u}$ is proven to be not invariant under $\Theta_{n}$ if there are $\bar{j}_1, \bar{j}_2 \in \Z_q$ such that
\begin{equation}\label{eq: CondCirculant} \langle Q(u^{\odot d})^{\odot d},T_{\bar{j}_1}(u)^{\odot d} \rangle \neq \langle Q(u^{\odot d})^{\odot d},T_{\bar{j}_2}(u)^{\odot d} \rangle. \end{equation}
In this case we may assume that $\bar{j}_2 = \bar{0}$ finishing the proof of the Lemma \ref{lem: originAndDistance}. \end{proof} 
\begin{proof}[Proof of Lemma \ref{lem: TaylorExp}]
\label{pr: lemTaylorExp}
As $W_\tau$ is a $\mathcal{C}^\infty$-manifold, there is a $\mathcal{C}^\infty$-function $h$ defined on a neighborhood $V  \subset \text{T}_{eq}W_\tau$ of $\vec{0} $ such that the assignment $u(v):=eq +v+h(v)$ maps $V$  onto some neighborhood $U \subset W_\tau$ of $eq$. Moreover, the differential $\text{D}h[\vec{0}]$ vanishes.
We will perform a second-order Taylor expansion of the real function \[v \mapsto \langle Q(u(v)^{\odot d})^{\odot d}, (T_{\bar{j}} u(v))^{\odot d}-u(v)^{\odot d}\rangle\] around $\vec{0} \in \text{T}_{eq}W_\tau$.
	
First note that for any vector $w \in \mathbb{R}^q$ and $\vec{1}=\begin{pmatrix} 1, & 1, &\ldots &1 \end{pmatrix}^T \in \mathbb{R}^q$ we have $\langle \vec{1},T_{\bar{j}}w-w \rangle=0$ for all $\bar{j} \in \Z_q$. 
Hence, both the constant term and the first-order term in the expansion will vanish and only the mixed term in the second derivative remains. We have
\begin{equation*}
\text{D}(u^{\odot d})[\vec{0}](v)=dq^{1-d}(v+\text{D}h[\vec{0}](v))=dq^{1-d}v
\end{equation*}
and
\begin{equation*}
\begin{split}
\text{D}(Q^q(u^{\odot d})^{\odot d})[\vec{0}](v)&=dQ^q(eq^{\odot d})^{\odot d-1} \odot Q^q(\text{D}(u^{\odot d}){[\vec{0}]}(v))
\cr 
&=d^2q^{1-d^2}\Vert Q \Vert_1 Q^q(v).
\end{split}
\end{equation*}
Thus, on the neighborhood $V \subset \text{T}_{eq}W_\tau$ of $eq$ it follows
\begin{equation}\label{eq: QuadExp}
\begin{split}
&\langle Q^q(u(v)^{\odot d})^{\odot d}, (T_{\bar{j}} u(v))^{\odot d}-u(v)^{\odot d}\rangle \cr 
&=\frac{1}{2}\text{D}^2(\langle Q^q(u(v)^{\odot d})^{\odot d} \;, \; (T_{\bar{j}} u(v))^{\odot d}-u(v)^{\odot d}\rangle)[\vec{0}](v,v)+O(\Vert v \Vert_1^3)\cr 
&=\frac{1}{2}2 \langle \  \text{D}(Q^q(u^{\odot d})^{\odot d})[\vec{0}](v) \; , \;\text{D}(T_{\bar{j}}u^{\odot d})[\vec{0}](v)- \text{D}(u^{\odot d})[\vec{0}](v) \ \rangle+O(\Vert v \Vert_1^3) \cr 
&=d^3q^{2-d(d+1)}\Vert Q \Vert_1 \langle Q^q(v),T_{\bar{j}}v-v \rangle+O(\Vert v \Vert_1^3). 
\end{split}
\end{equation} 
%	where the constant $c_2>0$ in the inequality is uniform in the direction $v$. This follows from continuity of the third derivative,  compactness and possibly shrinking the domain $V$.
	
To estimate $\langle Q^q(v),T_{\bar{j}}v-v \rangle$, we may sum over all $\bar{j} \in \Z_q$. 
Then 
\begin{equation*}
\sum_{\bar{j} \in \Z_q}\langle Q^q(v),T_{\bar{j}}v-v \rangle=-q \langle Q^q(v),v \rangle,
\end{equation*}
where the first term vanishes as $v \in \text{T}_{eq}\Delta^q$.
	
Hence 
\begin{equation} \label{eq: ellipticEst}
\begin{split}
\max_{\bar{j} \in \Z_q}\vert\langle 
Q^q(v) , T_j  v- v\rangle \vert &\geq \vert \langle 
Q^q(v) , v\rangle \vert.
\end{split}
\end{equation}
	
Now assume that $\text{T}^{+(-)}_{eq}W_\tau$ has positive dimension. Then $v \in V \cap \text{T}^{+(-)}_{eq}W_\tau$ implies that the r.h.s of \eqref{eq: ellipticEst} is bounded from below by $\tau \Vert v \Vert^2_{1}$.	
Combining \eqref{eq: QuadExp} and \eqref{eq: ellipticEst} thus shows \eqref{eq: TaylorExp}.
This conludes the proof of Lemma \ref{lem: TaylorExp}.
\end{proof}

\begin{proof}[Proof of Proposition \ref{pro: Id}]\label{pr: proId}
Let $\lambda^s \in \mathbb{Z}_s^{\vec{L}}$ and $l^t \in \mathbb{Z}_t^{\vec{L}}$ denote boundary laws obeying the recursions \eqref{eq: RecIn} and \eqref{eq: RecAway} constructed via backwards iteration on the respective local unstable manifolds such that the condition of Lemma \ref{lem: TaylorExp} is fulfilled. Let $q=st$.
Define $\tilde{\lambda}^q \in {[0,\infty)^{\mathbb{Z}_q}}^{\vec{L}}$ and $\tilde{l}^q \in {[0,\infty)^{\mathbb{Z}_q}}^{\vec{L}}$ as the $q$-periodic continuations of $\lambda^s$ and $l^t$, i.e. at any edge $(x,y) \in \vec{L}$ the vector $\tilde{\lambda}_{xy}^q \in [0,\infty)^{\mathbb{Z}_q}$ is defined as the periodic continuation of the vector $\lambda_{xy}^s \in [0,\infty)^{\mathbb{Z}_s}$ and similar for the vector $\tilde{l}^q$.

By construction it follows that $\tilde{\lambda}^q$ and $\tilde{l}^q$ are boundary laws for the fuzzy transfer operator $Q^q$. Moreover,  $\nu^{\tilde{\lambda}^q}=\nu^{\lambda^s}$ and $\nu^{\tilde{l}^q}=\nu^{l^t}$.
Let $x \in \partial \{\rho\}$. If $\nu^{\tilde{\lambda}^q}=\nu^{\tilde{l}^q}$ then from the marginals' representation of Lemma \ref{lem: BondMarginal} at the edge $\{\rho,x\}$ we obtain    
\begin{equation}\label{eq: period}
\langle \tilde{\lambda}^q_{x\rho}, T_{\bar{j}}(\tilde{\lambda}^q_{\rho x}) \rangle=C\langle \tilde{l}^q_{x \rho}, T_{\bar{j}}(\tilde{l}^q_{\rho x}) \rangle
\end{equation}
for any $j \in \Z$ where the constant $C>0$ is independent of $j$ and $\bar{j} \in \Z_q$ denotes the mod-$q$ projection.
As $\tilde{\lambda}^q$ is the periodic continuation of an $s$-periodic vector, the l.h.s. of \eqref{eq: period} is an $s$-periodic function in $j \in \Z$. Similarly, the r.h.s is a $t$-periodic function in $j$. Now $s$ and $t$ were assumed to be coprime, hence the l.h.s is a constant function in $j$. 
In particular, for the boundary law $\lambda^s$ we necessarily have
\begin{equation*}
\langle \tilde{\lambda}^s_{x \rho}, T_{\bar{j}}(\tilde{\lambda}^s_{\rho x})-\tilde{\lambda}_{\rho x}^s \rangle=0
\end{equation*}
at any $\bar{j} \in \Z_s$ which is excluded by the Lemma \ref{lem: TaylorExp}. Hence the measures $\nu^{\lambda^s}$ and $\nu^{l^t}$ must have different marginals along the edge $\{\rho,x\}$ which concludes the proof of Proposition \ref{pro: Id}.
\end{proof}	
\begin{proof}[Proof of Theorem \ref{thm: existence}] \label{pr: thm4}
Recall that the map $S_q$ describes the boundary law equation \eqref{eq: perGenBL} under the assumption of radial symmetry.
	
If the condition \ref{cond: existence} holds true then Proposition \ref{pro: Eigenvalues} describing the spectrum of $\text{D}S_q[eq]$ and Theorem \ref{thm: Chaperon} ensure existence of a local unstable manifold $W_\tau$ for $S_q$ near the equidistribution $eq$. As described in Section \ref{subsec: recEq}, for any starting value $u \in W_\tau$  we then obtain a radially symmetric nonhomogeneous boundary law solution via backwards iteration of the map $S_q$.

By Theorem \ref{thm: Zachary} this boundary law solutions corresponds to a $\Z_q$-valued Markov chain Gibbs measure which by Theorem \ref{thm: DLR-eq} can be mapped to an integer valued gradient Gibbs measure.

At last, Proposition \ref{pro: non-translation} guarantees that lack of invariance under translations of the tree of the constructed boundary law solution carries over to the so-obtained gradient Gibbs measure for uncountably many choices of the starting value $u \in W_\tau$.

This concludes the proof of Theorem \ref{thm: existence}.
\end{proof}
\subsection{Proofs for Section \ref{sec: Applications}}

\begin{proof}[Calculation of Table \ref{fig: theModels}] \label{pr: Fig2}
\mbox{}\\	
\begin{description}
\item[SOS-model:] 	
The calculation of the Fourier transform is elementary as it only involves geometric series. The value of the fraction $\hat{Q}_{\text{SOS},\beta}(\pi)/\hat{Q}_{\text{SOS},\beta}(0)$ then follows immediately. 
In the next step, we calculate the fuzzy transfer operator $Q_{\text{SOS},\beta}^q$. Let $\bar{i} \in \Z_q$ and $i \in \{0, \ldots, q-1\} \cap \bar{i}$. Then we have
\begin{equation*}
\begin{split}
Q_{\text{SOS},\beta}^q(\bar{i}) &=\sum _{j \in \Z}\exp(-\beta \vert i+qj \vert) =\frac{\cosh(\beta(i-\frac{q}{2}))}{\sinh(\beta \frac{q}{2})} 
\end{split}
\end{equation*}	
where in the last step we have used decomposition into geometric series.
\item[Inverse-square model:] To verify the expression for the Fourier-transform of the inverse-square transfer operator 
we compute  its backwards Fourier-transform. For any nonzero integer $j$ 
we have the elementary integral relation $\int_0^{\pi} \cos(k j) (3 k^2 - 6 \pi k + 2 \pi^2) \text{d}k = (6 \pi)/j^2$ 
which is obtained from partially integrating twice.  
From this the result follows.   
				
It remains to calculate the fuzzy transfer operator $Q_{\text{ISq},a}^q$. 
		
Let $\bar{i} \in \Z_q$ and $i \in \{0, \ldots, q-1\} \cap \bar{i}$. Then we have
\begin{equation*}
\begin{split}
Q^q_{\text{ISq},a}(\bar{i})
&=\sum_{j \in \Z}Q_{\text{ISq},a}(i+qj) \cr
&=Q_{\text{ISq},a}(i)+\frac{\beta}{q^2}\sum_{j=1}^{\infty}{ \frac{1}{(j+\frac{i}{q})^2}}+\frac{\beta}{q^2}\sum_{j=1}^{\infty}{\frac{1}{(j-\frac{i}{q})^2}}
\end{split}
\end{equation*}	
which proves the claim.
\end{description}
\end{proof}
\newpage
\begin{proof}[Proof of Theorem \ref{thm: EB}]
\mbox{}\\
\begin{description}
\item[SOS-model:] 
For the SOS-model at any $\beta>0$ we have 
\begin{equation}\label{pr: thmEB}
\frac{\hat{Q}_{\text{SOS},\beta}(k)}{\hat{Q}_{\text{SOS},\beta}(0)}=\frac{e^{2\beta}-2e^\beta+1}{e^{2\beta}-2e^\beta \cos{k}+1}
\end{equation}
which is strictly decreasing with $k$ on $[0,\pi]$.
Moreover, the expression is strictly positive.

Now the condition $\hat{Q}_{\text{SOS},\beta}(k)>(1/d)\hat{Q}_{\text{SOS},\beta}(0)$ of Theorem \ref{thm: existence}
is equivalent to 
\begin{equation} \label{eq: ConcrBoundsSOS}
\cos{k}>d-(d-1)\cosh(\beta).
\end{equation}
From the equation \eqref{eq: ConcrBoundsSOS} we obtain the following:

First, inserting the value $k=\pi$ we arrive at the lower bound $\beta=\arcosh(\frac{d+1}{d-1})$ for existence of non-invariant gradient Gibbs measures (n-t.i. GGM) of all periods $q$.

Second, assuming $\beta$ below this threshold and substituting $k=\frac{2\pi}{q}$, we obtain the minimal period $q(\beta,d)$ for the SOS-model.
\item[Inverse-square model:]
On the other hand, for the inverse square model at parameter $a >0$ we have
\begin{equation} \label{eq: ConcrBoundsIsq}
\frac{\hat{Q}_{\text{ISq},a}(k)}{\hat{Q}_{\text{ISq},a}(0)}=\frac{\frac{a}{2}(k-\pi)^2+1-\frac{a \pi^2}{6}}{1+\frac{a \pi^2}{3}},
\end{equation}
which is strictly decreasing with $k$ on $[0,\pi]$ and has a zero if and only if $a \geq \frac{6}{\pi^2}$.

Depending on the value of the expression \eqref{eq: ConcrBoundsIsq} at $k=\pi$ there are three different regions for the parameter $a$:

If $a <\frac{6}{\pi^2}\frac{d-1}{d+2}$ then $\frac{\hat{Q}_{\text{ISq},a}(\pi)}{\hat{Q}_{\text{ISq},a}(0)}>\frac{1}{d}$. Hence, by monotonicity $\frac{\hat{Q}_{\text{ISq},a}(k)}{\hat{Q}_{\text{ISq},a}(0)}>\frac{1}{d}$ for all $k$ meaning that existence of n-t.i. GGM of any period $q$ is guaranteed by Theorem \ref{thm: existence}.

If $\frac{6}{\pi^2}\frac{d-1}{d+2} \leq a \leq \frac{6}{\pi^2}\frac{d+1}{d-2}$ then $-\frac{1}{d}  \leq \frac{\hat{Q}_{\text{ISq},a}(\pi)}{\hat{Q}_{\text{ISq},a}(0)} \leq \frac{1}{d}$. 
%$\frac{\hat{Q}_{\text{ISq},a}(k)}{\hat{Q}_{\text{ISq},a}(0)}>-\frac{1}{d}$ for all $k \in [0,\pi]$.
This means  that $\vert \frac{\hat{Q}_{\text{ISq},a}(k)}{\hat{Q}_{\text{ISq},a}(0)} \vert >\frac{1}{d}$ is only possible if $\frac{\hat{Q}_{\text{ISq},a}(k)}{\hat{Q}_{\text{ISq},a}(0)} >\frac{1}{d}$. Solving this inequality for $k\in [0,\pi]$ gives the condition 
\begin{equation} \label{eq: ISqPosEV}
k<\pi-\sqrt{\frac{\pi^2}{3}(1+\frac{2}{d})-\frac{2}{a}(1-\frac{1}{d})}.
\end{equation} Substituting $k=\frac{2\pi}{q}$, we obtain the minimal period $q(a,d)$ for the Inverse square-model in that case.

Finally, if $a>\frac{6}{\pi^2}\frac{d+1}{d-2}$ then $\frac{\hat{Q}_{\text{ISq},a}(\pi)}{\hat{Q}_{\text{ISq},a}(0)} < -\frac{1}{d}$. In this case, existence of $2$-height periodic n-t.i. GGM is guaranteed and also those (positive) parts of the spectrum of $\text{D}S_q$ satisfying \eqref{eq: ISqPosEV} will correspond to height periodic non-t.i. GGM. Moreover, $\frac{\hat{Q}_{\text{ISq},a}(2\pi/3)}{\hat{Q}_{\text{ISq},a}(0)} < -\frac{1}{d}$ if and only if $a>\frac{9}{\pi^2}\frac{d+1}{d-3}$. In this case, also $q$-height periodic n-t.i. GGM for  $q=3$ and hence, by monotonicity, for all $q \geq 3$ exist.
\end{description}	
\end{proof}
%\newpage
\printbibliography
\end{document}